%% file: main-preprint.tex
\documentclass[10pt]{article}
\usepackage[letterpaper,margin=1in]{geometry}
\usepackage[utf8]{inputenc}
\usepackage[T1]{fontenc}
\usepackage{amsmath}
\usepackage{amssymb}
\usepackage{amsthm}
\usepackage{amscd}
\usepackage{graphicx}
\usepackage{rotating}
\usepackage{float}
\usepackage{booktabs}
\usepackage{enumitem}
\usepackage{cancel}
\usepackage{url}
\providecommand{\href}[2]{#2}
\providecommand{\texorpdfstring}[2]{#1}
\input{comandos}

\renewcommand{\C}{\bb C}
\DeclareMathOperator{\diam}{diam}
\newtheorem{theorem}{Theorem}[section]
\newtheorem{proposition}[theorem]{Proposition}
\newtheorem{lemma}[theorem]{Lemma}
\newtheorem{corollary}[theorem]{Corollary}

\newtheorem{remark}[theorem]{Remark}
\newtheorem{definition}[theorem]{Definition}

\input{enviroments}

\title{Hierarchical Expansion of Finite Discrete Dynamical Systems: A Non-Archimedean Variational Theory over Coordinate Orderings}

\author{J.~Rogelio P\'erez-Buend\'ia\thanks{SECIHTI--CIMAT, Unidad M\'erida, M\'erida, Yucat\'an, Mexico. Corresponding author: \texttt{rogelio.perez@cimat.mx}. ORCID 0000-0002-7739-4779.}
\and V\'ictor Nopal-Coello\thanks{Universidad Aut\'onoma del Estado de Hidalgo, Pachuca, Hidalgo, Mexico. \texttt{victor\_nopal@uaeh.edu.mx}. ORCID 0000-0003-2608-3636.}}
\date{}

\begin{document}

\maketitle

\begin{abstract}
A finite dynamical system on $\mathbb{F}_p^N$ has a functional graph and combinatorial basins, but nothing on which to run the local analysis a multiplier makes possible. We supply one. Orderings encode configurations as residue classes in $\mathbb{Z}_p$, and an exact rational interpreter over $\mathbb{C}_p$ carries each ball onto its image ball. Rational approximants send, resolution by resolution, every ball onto one of prescribed radius and center. A half-integer pole sphere, with no Archimedean counterpart, settles the expanding case. The prescribed radii are the missing multipliers. They sort balls into contracting, expanding or isometric, giving each configuration a word over three letters, one per scale. Fixed configurations all satisfy $f(a)=a$, yet their words refine them into scale-resolved classes. Contracting balls containing their images trap unique attracting fixed points of the approximant, expanding balls inside their images unique repelling ones. The counts become scores $(\mu_E,\mu_A,\mu_I)$, Haar integrals over $\mathbb{Z}_p$, and minimizing $\mu_E$ over orderings is the finite variational principle. The first score vanishes exactly on Anashin's $p$-adic automata, grading departures by resolution, while the ceiling $(N-1)p^N$ at every ordering isolates the $28$ outer-permutive rules among elementary cellular automata on rings of $N\ge5$ cells. The transition table and ordering determine radii, letters and scores. On a thirteen-gene Boolean network we certify the minimizer over all $13!$ orderings, whose hierarchy separates the ten fixed states by type.
\end{abstract}

\par\medskip

\input{01_introduction}
\input{02_setup}
\input{04_hierarchical_framework}
\input{05_toy_example}
\input{06_stability_measure}

\input{06b_structure_functional}
\input{07_application_athaliana}
\input{08_discussion}

\section*{ACKNOWLEDGMENTS}
The authors are affiliated with the P-ADAGIO research group (P-adic Arithmetic, Dynamics And Galois-Informed Observations), and thank its members for their support. We thank Esteban S\'anchez Dur\'an for his contributions in the early stages of this project. The authors acknowledge the Supercomputing Center of CIMAT-M\'erida ``TOOLOK'' for computing time.

\section*{FUNDING}
This work was supported by the project ``Modelos matem\'aticos y computacionales no convencionales para el estudio y an\'alisis de problemas relevantes en Biolog\'ia'' (CF 2019/217367), funded by SECIHTI.

\section*{CONFLICT OF INTEREST}
The authors declare that they have no conflicts of interest.

\section*{SUPPLEMENTARY INFORMATION}
Supplementary material accompanies this article. It contains the saturation theorem for elementary cellular automata and its proof, the target-ball data for the hierarchical approximants of the \textit{A.\ thaliana} network, the worked toy model with the target-ball specifications of its approximants and the gluing scheme that produces them, the count of the maps in the vanishing locus, the biological reading of the \textit{A.\ thaliana} example, the consistency estimate for consecutive approximants, and the $A/E/I$ landscape. The transition tables, search logs and scripts are in the reproducibility bundle~\cite{padic_grn_bundle}.

\section*{DATA AND CODE AVAILABILITY}
The transition table of the network studied here, the code that computes $\mu_E$, $\mu_A$ and $\mu_I$, the exact search over $S_N$, and a set of verification scripts are deposited in the versioned Zenodo record~\cite{padic_grn_bundle}, archived separately from the submission files. The verification scripts recompute the three scores and their control identity, reproduce the counterexample of Remark~\ref{rem:submodularity_negative}, extract the dependency digraph of Section~\ref{sec:athaliana} from the transition table, recover the minimum over $S_{13}$ by a second exact route that does not enumerate orderings, and carry out the survey of the $256$ elementary rules behind Theorem~\ref{thm:eca_saturation_main}. No other data were generated or analyzed.

\bibliographystyle{plain}
\bibliography{references}

\end{document}

%% file: comandos.tex
\newcommand{\bb}[1]{\mathbb{#1}}

\providecommand{\C}{\bb C}

\newcommand{\Qp}{{\bb{Q}}_{p}}

\mathchardef\mslash="202F

\def\DCP{\mathfrak{D}}

\def\BCP{\mathfrak{B}}
\def\BZP{{B}}

%% file: enviroments.tex
\makeatletter
\@ifundefined{Conjecture}{}{}
\@ifundefined{Problem}{}{}
\@ifundefined{Question}{}{}
\makeatother

%% file: 01_introduction.tex
\section{Introduction}
\label{sec:introduction}

Iteration of maps on $\mathbb{Z}_p$ that respect the $p$-adic hierarchy is a developed subject, with the $1$-Lipschitz maps, equivalently the functions computed by $p$-adic automata, at its center~\cite{Anashin2009,anashin2021automata_chapter}. This paper places a finite discrete dynamical system $f\colon\mathbb{F}_p^N\to\mathbb{F}_p^N$ inside that setting, and the passage requires a choice. Reading the $N$ coordinates in one order rather than another sends the same map into a different multiscale geometry, while scalar invariants such as entropy or average sensitivity stay fixed under every permutation of the variables. Many finite systems are hierarchical, with some coordinates determining coarse alternatives and others refining outcomes already selected. We make the ordering a variable of the problem and measure expansion across scales.

The construction starts with a choice of ordering. It embeds $\mathbb{F}_p^N$ into $\mathbb{Z}_p$ by base-$p$ expansion, so that a configuration becomes the ball of all $p$-adic integers sharing its first $N$ digits. Agreement in the first $n$ ordered coordinates means membership in a common ball of radius $p^{-n}$. Changing the ordering changes which variables control the coarsest balls, and hence the multiscale geometry in which the same finite map is observed.

The non-Archimedean layer then realizes this hierarchy analytically. Using rational dynamics over $\mathbb{C}_p$ we interpret the finite map on its configuration balls and prove that, at each resolution, there is a rational approximant $F_n\in\mathbb{C}_p(z)$ whose image of every source ball has the center and the radius prescribed by the transition data alone (Theorem~\ref{thm:existence_approximations}). The construction returns a resolution-dependent family, and the freedom in the choice of interpreter is what makes that family possible. Each $F_n$ realizes exactly the transition information visible at scale $n$, and successive approximants resolve the hierarchy at finer scales, so the same configuration is observed at $N$ nested resolutions and carries a type at each one. The expanding case is where the argument becomes genuinely non-Archimedean: a rational localization factor places its poles on a sphere of half-integer radius, which misses $\mathbb{Z}_p$.

This realization turns combinatorial structure into an analytic register. The functional graph of a finite system partitions its state space into the combinatorial basins of its periodic orbits, recording where each configuration eventually lands. The ball hierarchy supplies complementary information: it measures how much of the image is already determined by the first $n$ ordered coordinates. From the transition map and the embedding we define, for each ball, a common-prefix exponent $M_{n,m}$, an image radius $t_{n,m}=p^{-M_{n,m}}$, and a coarse multiplier $\Lambda_{n,m}=p^{n-M_{n,m}}$, the ball-level analogue of $|\phi'(x_0)|_p$. These quantities classify every ball as expanding, contracting, or isometric.

On nested ball chains the coarse classification becomes local fixed-point dynamics. A contracting ball containing its own image carries a unique attracting fixed point of the corresponding approximant, and every point of the ball converges to it, so the ball is a trapping neighborhood inside the analytic basin of that point. An expanding ball contained in its own image carries a unique repelling fixed point. In both cases the coarse multiplier is the exact multiplier at that fixed point, $|F_n'(b)|_p=\Lambda_{n,m}$, which is what makes $\Lambda_{n,m}$ the ball-level analogue it was introduced to be. If an isometric ball is carried onto itself, the approximant acts on it as an isometry and every fixed point it contains is indifferent, while existence depends on the map.

For a fixed configuration $a$ the nesting is automatic: the source ball and its image both contain the microball labeled by $a$, so they meet, and ultrametricity places them in one of the three cases above. Reading the label of that ball at every resolution produces a word over the three types, one letter per scale. The word is a scale-resolved invariant of the ball chain of $a$, and it distinguishes fixed configurations that the equation $f(a)=a$ alone leaves indistinguishable. It is computed from $(f,\iota)$, while Theorem~\ref{thm:existence_approximations} realizes its letters as image-radius ratios of rational maps on non-Archimedean balls. In the trapping and covering cases the attracting and repelling fixed points that accompany those letters may vary with the resolution and with the approximant chosen.

The weighted counts of the three classes define scores $(\mu_E,\mu_A,\mu_I)$, and minimizing $\mu_E$ over the symmetric group $S_N$ is a finite variational principle on coordinate hierarchies. The weight $p^{N-n}$ is $p^N$ times the normalized Haar mass of a level-$n$ ball, so the weighting comes from the $p$-adic hierarchy being measured, and it turns the score into a Haar integral, $\mu_E(\pi)=p^N\int_{\mathbb{Z}_p}\Phi^{E}_\pi\,d\lambda_{\mathrm{H}}$, whose integrand counts at each point the scales at which the lift fails to be $1$-Lipschitz. Vanishing of the functional is the classical automaton condition, stated as Theorem B below. Anashin developed the $p$-adic theory of these automata and characterized measure preservation and ergodicity inside the $1$-Lipschitz class~\cite[Theorem~5.6]{anashin2021automata_chapter}. What the functional adds to that equivalence is a graded form of it. Where that theory quantifies, it does so by a single exponent, the class $\mathrm{Lip}_\alpha$ of maps that lose at most $\alpha$ digits at every scale, refined in several variables to one exponent per variable~\cite{leoncardenal2022lipschitz}, and such an exponent bounds the largest loss a map ever produces. The functional $\mu_E$ records instead at which resolutions and on which balls the condition fails. Minimizing therefore asks how close a finite system can be brought to a $p$-adic automaton by reordering its coordinates, and the minimum is $0$ exactly for the systems whose dependency digraph is acyclic once loops are removed, with the topological orders as minimizers.

Taken together across all $N$ resolutions, these labels describe a landscape that Section~\ref{sec:structure_functional} makes precise. As a worked example we apply the variational principle to an archived $13$-variable Boolean transition map based on the floral regulatory network of \textit{Arabidopsis thaliana}, which was not built for the present purpose. The minimum $\mu_E=26{,}776$ is certified over all $13!$ orderings, against $80{,}704$ at the ordering in which the model is usually written. No gene label or prior ranking enters the computation, and it places first the one coordinate the dynamics conserves, then three documented regulators of the floral transition. At the optimum the ball chains of the ten fixed configurations split into two scale-resolved classes, six of them entering contracting level-$4$ balls and four entering expanding ones.

\paragraph{The three main results.}
Write $\widehat{F}_\pi$ for the lift of $f$ to $\mathbb{Z}_p$ determined by an ordering $\pi\in S_N$. The first result is the realization itself, and it holds for every finite system and every ordering.

\medskip
\noindent\textbf{Theorem A} (Rational realization and hierarchical approximation; Proposition~\ref{prop:exact_existence}, Theorem~\ref{thm:existence_approximations} and Proposition~\ref{Prop:AproxFnPhi})\textbf{.} \emph{Let $f\colon\mathcal{C}\to\mathcal{C}$ be any map and fix a coordinate ordering. There is a rational $\phi\in\mathbb{C}_p(z)$ carrying the ball of each configuration onto the ball of its image, and, for every resolution $n$, a rational $F_n\in\mathbb{C}_p(z)$ whose image of each level-$n$ ball is the ball of prescribed radius and center determined by the transition data alone. On the observable trace of that ball the two maps agree to within the radius of the prescribed image ball.}
\medskip

\noindent The two levels do different work. Projection from the configuration balls semiconjugates the interpreter to $f$ (Remark~\ref{rem:semiconjugacy}), so the finite system is the shadow of a rational one, while each $F_n$ realizes exactly the ball dynamics visible at resolution $n$. The multipliers, the $A/E/I$ labels and the functional are read off that family. The second result says what the zero locus of the functional is.

\medskip
\noindent\textbf{Theorem B} (Theorem~\ref{thm:muE_zero_lipschitz})\textbf{.} \emph{For a map $f$ and an ordering $\pi$, the following are equivalent: $\mu_E(\pi)=0$; no ball of the hierarchy expands; $f$ is triangular for $\pi$; and $\widehat{F}_\pi$ is $1$-Lipschitz, that is, an automaton function on the alphabet $\mathbb{F}_p$.} The full statement lists six equivalent conditions.
\medskip

\noindent The third determines the opposite extreme, where no ordering helps at all, and identifies the elementary cellular automata that reach it.

\medskip
\noindent\textbf{Theorem C} (Theorems~\ref{thm:muE_saturation} and~\ref{thm:eca_saturation_main})\textbf{.} \emph{One has $\mu_E(\pi)\le(N-1)p^N$ for every ordering, and equality holds for every $\pi\in S_N$ if and only if no set of coordinates $J$ with $1\le|J|\le N-1$ is autonomous at any $u\in\mathbb{F}_p^{J}$. Among the elementary cellular automata on a ring of $N\ge5$ cells these are exactly the $28$ rules that are permutive in an outer variable.}
\medskip

\noindent Theorem A equips every chosen coordinate hierarchy with rational realizations and ball-level multipliers. Theorem B identifies the zero endpoint of the resulting expansion scale with the classical automaton condition, and Theorem C identifies its saturation endpoint and the finite maps that occupy it.

\paragraph{Contributions and relation to prior work.}
The exact interpreter of the first assertion of Theorem~A is supplied by the gluing construction of~\cite{rogelio2023gluing}. The resolution-dependent family $F_0,\dots,F_N$ is developed here, with the center and the radius of every image ball prescribed simultaneously from the finite transition data, and with the expanding case handled by the half-integer pole sphere already described. Those radii produce the ball-level multipliers $\Lambda_{n,m}$, the $A/E/I$ profiles, the functional $\mu_E$ with its Haar representation, and the variational problem over $S_N$. In Theorem~B the equivalence of $1$-Lipschitz maps, triangular maps and automaton functions is Anashin's~\cite{Anashin2009,anashin2021automata_chapter}. What is added here is that this condition is the zero locus of a functional grading its failure resolution by resolution, together with Corollary~\ref{cor:min_zero_dag}, which identifies the existence of a zero-score ordering with acyclicity of the dependency digraph and returns every minimizer as a topological order of it. What the minimization costs, and under which presentation of $f$, is taken up in Section~\ref{sec:discussion}. Outer permutivity is classical~\cite{hedlund1969endomorphisms}, and Theorem~C establishes the universal bound, characterizes its saturation by the absence of nontrivial coordinate subsets autonomous at even one partial state, and shows that for the elementary rules that criterion selects precisely $28$ of the $256$. Between the two endpoints, Proposition~\ref{prop:intermediate_constant} characterizes when the expansion score is independent of the ordering: $\mu_E$ is constant on $S_N$ if and only if the nonautonomy count of a coordinate set depends only on its cardinality. It also exhibits a sharp constant intermediate regime, so order-independent expansion scores occur at the minimum, at the maximum, and at genuine values between them.

The paper is organized as follows. Section~\ref{sec:setup} fixes the non-Archimedean setup, the coarse embedding into $\mathbb{Z}_p$, and the closed-ball bridge that relates trapping, covering and isometric ball dynamics to attracting, repelling and, when they exist, indifferent fixed points. Section~\ref{sec:hierarchical_framework} constructs rational interpreters, hierarchical approximants with canonical target balls, and the ball-level classification. Section~\ref{sec:toy_example} runs a small example, with details in the Supplementary Material. Section~\ref{sec:stability_measure} defines the scores and states the variational problem, and Section~\ref{sec:structure_functional} proves what the expansion functional measures, from its Haar integral form to its vanishing and saturation loci. Section~\ref{sec:athaliana} works out the \textit{A. thaliana} example and Section~\ref{sec:discussion} discusses scope and open problems.

%% file: 02_setup.tex
\section{Non-Archimedean Setup and the Coarse Embedding}
\label{sec:setup}
\label{sec:preliminaries}

We work throughout over $\mathbb{C}_p$, the completion of an algebraic closure of $\mathbb{Q}_p$, which plays for $p$-adic dynamics the role that $\mathbb{C}$ plays in the complex theory. We assume the standard non-Archimedean vocabulary and recall only what the construction below uses. Standard references are~\cite{benedetto2019dynamics,Zuniga-Galindo2022,zunigagalindo2025padicbook}. We denote by $\BCP_r(a)$ the closed ball of radius $r$ about $a$ in $\mathbb{C}_p$ and by $\DCP_r(a)$ the open disk, in fraktur. Their traces on the observable points $\mathbb{Z}_p$ are written with an italic $B$, as fixed in Remark~\ref{rem:ball_notation} below. When $r\notin p^{\mathbb{Q}}$ the open and closed balls of radius $r$ coincide. Recall one ultrametric fact, which we use quantitatively: if $\BCP_r(c)$ and $\BCP_s(d)$ are disjoint, then $|x-y|_p=|c-d|_p$ for all $x\in\BCP_r(c)$ and $y\in\BCP_s(d)$.

\begin{remark}[Notation: analytic versus observable balls]
\label{rem:ball_notation}
Two kinds of ball run through the paper and the typography is the only thing that separates them.
Fraktur denotes the analytic object and italic its observable trace:
\[
\BCP_r(a)\subset\C_p \quad\text{(analytic ball)},\qquad
\BZP_r(a):=\BCP_r(a)\cap\mathbb{Z}_p \quad\text{(observable ball)} .
\]
The $p^N$ micro-balls centered at the integer representatives cover $\mathbb{Z}_p$ and fall far short of covering a ball of $\C_p$, whose residue field is infinite. Approximate comparisons are stated on observable balls $\BZP$, exact image equalities on analytic balls $\BCP$.
\end{remark}

\subsection{Observable Points, Balls, and Images}

Among the points of $\mathbb{C}_p$ we single out $\mathbb{Z}_p$ as the \emph{observable points}, the states the model resolves. The \emph{observable ball} attached to a $\mathbb{C}_p$-ball $\BCP_r(a)$ is its trace $\BCP_r(a)\cap\mathbb{Z}_p$, written $\BZP_r(b)$ for any $b$ in it. For $x\in\mathbb{Z}_p$ and $n\ge 1$ the unique $m\in\{0,\dots,p^n-1\}$ with $x\in\BZP_{1/p^n}(m)$ collects the first $n$ digits of the expansion of $x$, so under $\mathbb{Z}_p=\varprojlim_k\mathbb{Z}/p^k\mathbb{Z}$ the ball $\BZP_{1/p^n}(m)$ is the residue class $x\equiv m\pmod{p^n}$ and carries mass $\lambda_{\mathrm{H}}(\BZP_{1/p^n}(m))=p^{-n}$ for the normalized Haar measure $\lambda_{\mathrm{H}}$ on $\mathbb{Z}_p$. The ordering chosen in Section~\ref{subsec:embedding} is what decides which coordinate controls which level of that hierarchy.

\paragraph{Observable image of a rational map.}
For a nonconstant $\phi \in \C_p(z)$ and a ball $\BCP_r(a)$ that meets $\mathbb{Z}_p$, two images must be distinguished. The \emph{pointwise image} $\phi\big(\BCP_r(a)\cap\mathbb{Z}_p\big)$ is the literal image of the observable points of the ball, and it need not consist of observable points, since $\phi$ may carry points off $\mathbb{Z}_p$. The \emph{observable image} $\phi\big(\BCP_r(a)\big)\cap\mathbb{Z}_p$ is the trace on $\mathbb{Z}_p$ of the image ball. By Proposition~\ref{prop:balls_to_balls} below, that image is a ball whenever it is not all of $\mathbb{P}^1(\C_p)$, which is the case as soon as $\phi$ has no pole on $\BCP_r(a)$ and the image is therefore bounded. The observable image is then again an observable ball whenever it is nonempty. The pointwise image lies inside the image ball, so the observable points it contains lie in the observable image, but the two sets differ in general. The constructions of Section~\ref{sec:hierarchical_framework} track the observable image, not the pointwise one.

\subsection{The Coarse Embedding}
\label{subsec:embedding}
\label{sec:discrete_model}

The framework applies to any finite system whose state space is identified with $\mathcal{C}=\mathbb{F}_p^N$ for a prime $p$: Boolean networks ($p=2$), cellular automata over a prime alphabet, polynomial endomorphisms of $\mathbb{F}_p^N$, and synchronous finite-state machines are all instances. We write $f\colon\mathcal{C}\to\mathcal{C}$ for the transition map. A choice of ordering of the $N$ coordinates encodes each configuration $a=(a_0,\dots,a_{N-1})\in\mathcal{C}$ as the integer center
\begin{equation}
\label{eq:padic_encoding}
m_a := \sum_{i=0}^{N-1} a_i p^i \in \{0,\ldots,p^N-1\},
\end{equation}
its base-$p$ value, an integer in $\{0,\ldots,p^N-1\}$. Reducing modulo $p^N$ sends $m_a$ to its class in $\mathbb{Z}/p^N\mathbb{Z}$, and the canonical isomorphism $\mathbb{Z}/p^N\mathbb{Z}\cong\mathbb{Z}_p/p^N\mathbb{Z}_p$ induced by $\mathbb{Z}\hookrightarrow\mathbb{Z}_p$ identifies that class with the coset
\begin{equation}
\label{eq:ball_association}
\iota(a) := m_a + p^N\mathbb{Z}_p \;=\; \BZP_{1/p^N}(m_a),
\end{equation}
the closed ball of radius $1/p^N$ centered at $m_a$. Two configurations land in the same ball if and only if their first $N$ ordered digits agree, so the encoding is faithful and its scales are exactly the coordinate hierarchy. The map $\iota\colon\mathcal{C}\xrightarrow{\ \sim\ }\mathbb{Z}_p/p^N\mathbb{Z}_p$
is therefore a bijection of the $p^N$ configurations onto the $p^N$ residue classes, equivalently the $p^N$ balls of radius $1/p^N$. We call it the \emph{coarse embedding} and write $\BZP(a) := \iota(a)$. A configuration is identified not with a point but with the class of all $p$-adic integers sharing its first $N$ digits. The embedding carries the dynamics, not merely the states: the rational interpreter $\phi$ of Section~\ref{sec:hierarchical_framework} realizes $f$ on these classes, mapping the ball of $a$ into the ball of $f(a)$ (Definition~\ref{def:interprets}), so each transition $a \mapsto f(a)$ of the functional graph lifts to an inclusion of balls.

Then $\mathbb{Z}_p = \bigsqcup_{a \in \mathcal{C}} \BZP(a)$: $p^N$ balls of radius $1/p^N$, one per configuration. The expansion~\eqref{eq:padic_encoding} encodes hierarchy: coordinates at lower powers of $p$ have greater influence on $p$-adic distance. The coordinate at the $p^0$-place is therefore the coarsest one in the ultrametric hierarchy, since it alone decides the partition of $\mathbb{Z}_p$ modulo $p$, even though it is the least significant digit in ordinary base-$p$ notation. The multiscale geometry depends on the chosen ordering of the coordinates. The role of this ordering and the variational principle over it are developed in Sections~\ref{sec:hierarchical_framework}--\ref{sec:stability_measure}.
\subsection{Rational Maps on Balls}

The dynamics over $\mathbb{C}_p$ is carried by rational functions, which act on balls. Let $\phi\in\C_p(z)$ be a rational function.
\begin{proposition}[Proposition 3.25 in \cite{benedetto2019dynamics}]
\label{prop:balls_to_balls}
Let $\BCP$ be a ball in $\C_p$ and $\phi \in \C_p(z)$ nonconstant. Then $\phi(\BCP)$ is either $\mathbb{P}^1(\C_p)$ or a ball in $\mathbb{P}^1(\mathbb{C}_p)$ of the same type as $\BCP$, and in the latter case every point of $\phi(\BCP)$ has the same number of preimages in $\BCP$, counting multiplicity.
\end{proposition}

Thus a nonconstant rational function maps a ball onto a ball, unless it maps that ball onto all of $\mathbb{P}^1(\mathbb{C}_p)$. This underlies the interpretation of discrete dynamics in Section~\ref{sec:hierarchical_framework}.

We write $\lambda=\phi'(x)$ for the multiplier at a fixed point and call $x$ attracting, repelling, or indifferent according to $|\lambda|_p<1$, $>1$, or $=1$.

\begin{proposition}[Cases (a) and (c) of Theorem 4.18 in \cite{benedetto2019dynamics}]
\label{T418}
Let $U = \DCP_r(a)$ and let $\phi$ be given on $U$ by a convergent power series in $z-a$, as is the case for a rational function with no poles on $U$, and assume $d:=\operatorname{wdeg}_U(\phi-a)<\infty$.
\begin{itemize}
  \item[(a)] If $\phi(U) \subsetneq U$, there is a unique fixed point $b\in U$, it is attracting, and $\phi^{n}(x)\to b$ for every $x\in U$.
  \item[(c)] If $\phi(U) \supsetneq U$ and $d=1$, there is a unique fixed point $b\in U$, and it is repelling.
\end{itemize}
\end{proposition}

\noindent The lettering follows the source. The cited theorem also treats a second attracting case, in which $\phi$ maps $U$ onto itself with Weierstrass degree at least $2$, and an isometric case. Neither is restated: the first does not arise in Section~\ref{sec:hierarchical_framework}, and the second is reached directly from the scaling identity of Lemma~\ref{lem:closed_ball_bridge}.

\begin{remark}[The Weierstrass degree is finite wherever we invoke Proposition~\ref{T418}]
\label{rem:wdeg_finite}
The Weierstrass degree of $f=\sum_{k\ge0}c_k(z-a)^k$ on a disk of radius $r$ about $a$ selects an index attaining $\max_k|c_k|_p r^k$: the largest such index on a rational closed disk, and on an open or irrational disk the smallest, or $\infty$ when no index attains it \cite[Definition~3.8]{benedetto2019dynamics}. On a rational closed disk the maximum is always attained, so finiteness is at issue only in the second convention, and there it is exactly the requirement that the maximum be attained. The hypothesis is therefore not vacuous, and a power series with radius of convergence exactly $r$ can fail it. It costs nothing under the standing hypothesis of Lemma~\ref{lem:closed_ball_bridge}. If $F\in\C_p(z)$ has no pole on the closed ball $\BCP_r(a)$, then it has none on $\BCP_r(b)=\BCP_r(a)$ for any $b\in\BCP_r(a)$, the poles lying at distance greater than $r$ from $b$ as well, so the expansion of $F$ at $b$ converges on that ball and its coefficients satisfy $|c_k(b)|_p r^k\to0$. Every disk contained in $\BCP_r(a)$ has some $b\in\BCP_r(a)$ as a center and a radius $s\le r$, so $|c_k(b)|_p s^k\to0$ and the maximum defining the Weierstrass degree, taken at the center of that disk, is attained. The sub-disk need not be concentric with $a$: when $|a-a'|_p=r$ the disks $\DCP_r(a')$ and $\DCP_r(a)$ are disjoint, and it is $\DCP_r(a')$ that the proof of Lemma~\ref{lem:closed_ball_bridge}(b) uses. The same bound makes the maximum attained on the closed ball itself. The approximants built in Section~\ref{sec:hierarchical_framework} satisfy the hypothesis: the construction of Theorem~\ref{thm:existence_approximations} keeps the poles of $F_n$ off the closed ball $\BCP_{1/p^n}(m)$, those of the local piece glued there sitting on a sphere of radius $p^{\,1/2-n}>p^{-n}$ when that piece has poles at all, and a pole of $F_n$ in the ball being excluded by the image equality of Theorem~\ref{T1}. The degree is therefore finite at every level and for every ball.
\end{remark}

\noindent In part~(a) the strict inclusion is what plays the role of a contraction, and no metric hypothesis on $\phi$ itself is needed. The ball-level analogue of this classification is given in Definition~\ref{def:quasi_dynamics} (Section~\ref{subsec:quasi_dynamics}).

\noindent A subtle point: Proposition~\ref{T418} is stated for open disks $\DCP_r(a) \subset \mathbb{C}_p$, while the constructions of Section~\ref{sec:hierarchical_framework} live on closed balls $\BCP_r(a)$. The two sets coincide for irrational radii $r \notin p^{\mathbb{Q}}$, so the proposition transfers directly there, but the radii $r=p^{-n}$ that arise in our setting lie in $p^{\mathbb{Q}}$. The following lemma is the bridge. It is stated for an arbitrary rational function, since nothing in it depends on the construction that will produce the approximants. Both parts place the image ball in the chain of $\BCP_r(a)$, which is exactly what the same-ball-chain hypothesis of Remark~\ref{rem:fixed_points} supplies; in part~(b) that costs no generality, because $\BCP_r(a)\subseteq\BCP_t(b)$ already forces $\BCP_t(b)=\BCP_t(a)$. Part~(a) costs no generality either, though for a different reason. If $F(\BCP_r(a))=\BCP_t(\beta)\subsetneq\BCP_r(a)$, then $\beta\in\BCP_r(a)$, so $\BCP_r(\beta)=\BCP_r(a)$ as sets, and the hypothesis of~(a) holds at the center $\beta$ with $r'=t<r$. The change of center is not idle. When $|\beta-a|_p=r$, as for $\BCP_{1/4}(3)\subsetneq\BCP_{1/2}(1)$ with $p=2$, no ball concentric with $a$ of radius smaller than $r$ contains the image.

\begin{lemma}[Closed-ball bridge]
\label{lem:closed_ball_bridge}
Let $F\in\C_p(z)$ have no pole on the closed ball $\BCP_r(a)$, with $r\in p^{\mathbb{Q}}$.
\begin{itemize}
  \item[(a)] If $F(\BCP_r(a))\subseteq\BCP_{r'}(a)$ for some $r'<r$, then $F$ has a unique fixed point $b$ in $\BCP_r(a)$, it is attracting, and $F^k(x)\to b$ for every $x\in\BCP_r(a)$.
  \item[(b)] If $F$ carries $\BCP_r(a)$ bijectively onto $\BCP_t(a)$ for some $t>0$, then $F$ scales every distance in the ball by the same factor,
  \begin{equation}
  \label{eq:scaled_isometry}
  |F(x)-F(y)|_p=\frac{t}{r}\,|x-y|_p\qquad\text{for all }x,y\in\BCP_r(a).
  \end{equation}
  In particular $|F'(z)|_p=t/r$ throughout the ball, so the multiplier at any fixed point is $t/r$. If $t>r$, then $F$ has a unique fixed point in $\BCP_r(a)$ and it is repelling. If $t=r$, then $F$ is an isometry of $\BCP_r(a)$ and any fixed point it has is indifferent, but existence is not forced. If $t<r$, then $\BCP_t(a)\subsetneq\BCP_r(a)$ and part~(a) applies, so the unique fixed point is attracting with multiplier of absolute value $t/r<1$.
\end{itemize}
\end{lemma}

\begin{proof}
Throughout, $F$ is given on $\BCP_r(a)$ by a convergent power series $F(z)=\sum_{k\ge0}c_k(z-a)^k$. Write $F=g/h$ with $g,h\in\C_p[z]$ coprime, so that the absence of poles on the ball says exactly that $h$ has no zero there. Then $h$ is a unit in the ring of power series converging on $\BCP_r(a)$, and $F=g\,h^{-1}$ lies in that ring \cite[Theorem~3.5 and Proposition~3.11]{benedetto2019dynamics}. Its Weierstrass degree is finite on that ball and on every disk inside it (Remark~\ref{rem:wdeg_finite}). We may assume $F$ nonconstant. A constant map is ruled out in~(b), which asks for a bijection onto a ball of positive radius, and in~(a) it makes its own value the unique fixed point, attracting because the multiplier is $0$. Being bounded, the image of $F$ is a ball and not $\mathbb{P}^1(\C_p)$, so Proposition~\ref{prop:balls_to_balls} applies in its first alternative.

(a) Density of $p^{\mathbb{Q}}$ in $\mathbb{R}_{>0}$ supplies $s,\rho\in p^{\mathbb{Q}}$ with $r'<\rho<s<r$. Then $\DCP_s(a)\subset\BCP_r(a)$ and $F(\DCP_s(a))\subseteq\BCP_{r'}(a)\subseteq\DCP_s(a)$. The second of these inclusions is strict, because $\rho$ lies in the value group $|\C_p^\times|=p^{\mathbb{Q}}$, so some $x$ has $|x-a|_p=\rho$, and that $x$ lies in $\DCP_s(a)$ and outside $\BCP_{r'}(a)$. Hence $F(\DCP_s(a))\subsetneq\DCP_s(a)$, which is the hypothesis of Proposition~\ref{T418}(a), and that proposition gives a unique attracting fixed point in $\DCP_s(a)$. The intermediate radius $\rho$ is where the density of the value group enters. Over a field whose value group is discrete, such as $\Qp$, no such $\rho$ need exist and the two balls can agree. Any fixed point of $F$ in $\BCP_r(a)$ lies in the image, hence in $\BCP_{r'}(a)\subset\DCP_s(a)$, so it is that one. Proposition~\ref{T418}(a) also gives $F^k(y)\to b$ for every $y\in\DCP_s(a)$, and $F(\BCP_r(a))\subseteq\BCP_{r'}(a)\subset\DCP_s(a)$ puts every point of the closed ball inside that disk after a single iteration, so the convergence holds on all of $\BCP_r(a)$.

(b) Convergence on the closed ball gives $|c_k|_p r^k\to0$, so the maximum defining the Weierstrass degree is attained and $d:=\operatorname{wdeg}_{\BCP_r(a)}(F-c_0)$ is finite, and, $F$ being bijective and hence nonconstant, \cite[Theorem~3.15]{benedetto2019dynamics} says that $F$ is everywhere $d$-to-$1$ onto its image, a ball centered at $c_0=F(a)$ of radius $|c_d|_p r^d$. Bijectivity forces $d=1$: were $d\ge2$, the single preimage of each point of the image would carry the whole multiplicity, so $F-F(x)$ would vanish to order $d$ at every $x\in\BCP_r(a)$, hence $F'$ would vanish on that ball and therefore identically, a nonzero rational function having only finitely many zeros, and $F$ would be constant in characteristic zero, which a bijection onto a ball of positive radius is not. The image is therefore the closed ball of radius $|c_1|_p r$ about $F(a)$, and by hypothesis it is $\BCP_t(a)$. Equal balls need not have equal radii over an arbitrary non-Archimedean field, $\BCP_{1/2}(0)$ and $\BCP_{1/3}(0)$ being the same subset of $\Qp$ for $p=3$, but in $\C_p$ the density of $p^{\mathbb{Q}}$ makes the radius of a closed ball equal to its diameter and hence an invariant of the set. So $|c_1|_p r=t$ and $|c_1|_p=t/r$. Since $1$ is the largest index attaining the maximum, $|c_k|_p r^k<|c_1|_p r$ for every $k\ge2$. Hence for $x,y\in\BCP_r(a)$ every term of $F(x)-F(y)$ beyond the linear one is bounded by $|c_k|_p r^{k-1}|x-y|_p<|c_1|_p|x-y|_p$, and the strong triangle inequality leaves the linear term in charge, which is~\eqref{eq:scaled_isometry}. The same bound applied to $F'(z)=c_1+\sum_{k\ge2}kc_k(z-a)^{k-1}$, together with $|k|_p\le1$, gives $|F'|_p=|c_1|_p=t/r$ throughout the ball.

Suppose $t>r$. Since $a\in\BCP_r(a)\subset\BCP_t(a)=F(\BCP_r(a))$, there is $a'\in\BCP_r(a)$ with $F(a')=a$, and $\BCP_r(a')=\BCP_r(a)$ because every point of a non-Archimedean ball is a center. We claim
\begin{equation}
\label{eq:image_open_disk}
F\bigl(\DCP_r(a')\bigr)=\DCP_t(a)\supsetneq\BCP_r(a)\supsetneq\DCP_r(a'),
\end{equation}
where the second inclusion holds because $t>r$ together with the density of $p^{\mathbb{Q}}$, which puts a point at distance from $a$ strictly between $r$ and $t$, and the third because $\DCP_r(a')\subset\BCP_r(a')=\BCP_r(a)$, strictly because $r\in p^{\mathbb{Q}}$ provides a point at distance exactly $r$ from $a'$. The inequality $t>r$ on its own would not separate the two balls over a field with discrete value group, as $\DCP_1(0)=\BCP_{1/p}(0)$ in $\Qp$ shows. Only $\supseteq$ is needed in the second relation in any case, the strict containment that Proposition~\ref{T418}(c) will use coming from the third. For the equality, if $x\in\DCP_r(a')$ then $|F(x)-a|_p=|F(x)-F(a')|_p=(t/r)|x-a'|_p<t$ by~\eqref{eq:scaled_isometry}. Conversely, given $y\in\DCP_t(a)\subset\BCP_t(a)$, bijectivity supplies a unique $x\in\BCP_r(a')$ with $F(x)=y$, and~\eqref{eq:scaled_isometry} read backwards gives $|x-a'|_p=(r/t)|y-a|_p<r$, so $x\in\DCP_r(a')$.

Injectivity supplies the degree. By~\eqref{eq:image_open_disk} the restriction $F\colon\DCP_r(a')\to\DCP_t(a)$ is a bijection, so \cite[Theorem~3.15]{benedetto2019dynamics} applied on that open disk forces $\operatorname{wdeg}_{\DCP_r(a')}(F-a)=1$ by the multiplicity argument just given, the subtraction being of $F(a')=a$. Proposition~\ref{T418}(c) is stated instead for $\operatorname{wdeg}_U(\phi-a')$, subtracting the center of $U$. The two agree here, because the series differ by the constant $a-a'$, of absolute value at most $r<t$, while the maximum $t$ is attained by the linear term, so the constant cannot change which index attains it.

By~\eqref{eq:image_open_disk} the disk $\DCP_r(a')$ is strictly contained in its image, so Proposition~\ref{T418}(c) gives a unique repelling fixed point $b\in\DCP_r(a')$. Uniqueness extends to the whole closed ball by~\eqref{eq:scaled_isometry}: two distinct fixed points $x_0,x_1\in\BCP_r(a)$ would satisfy $|x_0-x_1|_p=(t/r)|x_0-x_1|_p$ with $t/r>1$. Finally $|F'(b)|_p=t/r>1$ by the first paragraph, so $b$ is repelling with the stated multiplier.

Suppose $t<r$. Then $F(\BCP_r(a))=\BCP_t(a)$, which is the hypothesis of part~(a) with $r'=t$, so there is a unique fixed point $b\in\BCP_r(a)$, it is attracting, and $|F'(b)|_p=t/r<1$ by the first paragraph. Suppose $t=r$. Then~\eqref{eq:scaled_isometry} says that $F$ is an isometry of $\BCP_r(a)$ and $|F'|_p=1$, so a fixed point, if there is one, is indifferent. Existence is not forced: the translation $z\mapsto z+1$ carries the closed unit ball $\BCP_1(0)$ bijectively and isometrically onto itself, since $\BCP_1(1)=\BCP_1(0)$, and has no fixed point anywhere in $\mathbb{C}_p$.
\end{proof}

\noindent Part~(b) transfers the expanding case of \cite[Theorem~4.18]{benedetto2019dynamics} to the closed balls on which Section~\ref{sec:hierarchical_framework} works, and strengthens it there: the descent to an open disk centered at a preimage of $a$ is what makes the cited theorem applicable, and the scaling identity is what returns the conclusion to the closed ball, where the fixed point is unique and the multiplier is exactly $t/r$. The chain~\eqref{eq:image_open_disk} runs through $\BCP_r(a)$, which is what makes the cited theorem apply without recentering the target disk. Compare \cite[Lemme~4.13]{riveraletelier2003dynamique}, which reaches the existence half for an arbitrary end.

\subsection{The Gluing Tool}

The construction of Section~\ref{sec:hierarchical_framework} builds a global rational map from prescribed local dynamics on disjoint balls. The tool is the following gluing theorem.
\begin{theorem}[$\varepsilon$-approximation; Theorem 4.2 in \cite{rogelio2023gluing}]
\label{T1}
Let $g_1, \ldots, g_n \in \C_p(z)$ and let $B := \bigcup_{i=1}^n \BCP_{r_i}(a_i)$, where
\begin{itemize}
    \item[(H1)] the $\BCP_{r_i}(a_i)$ are pairwise disjoint closed balls whose radii lie in the value group, $r_i\in|\C_p^\times|=p^{\mathbb{Q}}$ (such balls are called \emph{rational}, a property of the radius and not of the map);
    \item[(H2)] $g_i(\BCP_{r_i}(a_i)) = \BCP_{t_i}(b_i)$ for each $i$;
    \item[(H3)] $g_i(B) \subset \BCP_1(0)$ for all $i = 1, \ldots, n$.
\end{itemize}
Then for any $\varepsilon > 0$ there exists $F_\varepsilon \in \C_p(z)$ such that:
\begin{enumerate}
    \item\label{T1:image} $F_\varepsilon(\BCP_{r_i}(a_i)) = \BCP_{t_i}(b_i)$ for each $i$ (the image of each input ball equals the prescribed target ball in $\mathbb{C}_p$),
    \item $|F_\varepsilon(z) - g_i(z)|_p < \varepsilon$ for all $z \in \BCP_{r_i}(a_i)$.
\end{enumerate}
\end{theorem}

\noindent One explicit construction is
\[
F_\varepsilon(z) \;=\; \sum_i g_i(z)\, h_i(z),
\qquad
h_i(z) \;=\; \Bigl(1 - \bigl((z-a_i)/c_i\bigr)^{\kappa_i}\Bigr)^{-1},
\]
for appropriate $c_i$ and even exponents $\kappa_i$. Details are in~\cite[Section~4]{rogelio2023gluing}. Conclusion~\ref{T1:image} holds for every $\varepsilon>0$, and one step of its proof deserves to be made explicit here, because it is where ultrametric geometry bites: exhibiting points of the image at every admissible distance from $b_i$ does not by itself fill the ball, since a proper subball may meet the outer sphere of a larger one. The equality follows from the construction as follows. The poles of each factor $h_j$ lie on the sphere $|z-a_j|_p=\sqrt{\delta_j r_j}$, which is disjoint from every input ball, and each $g_j$ is pole-free on $B$ by hypothesis (H3), so $F_\varepsilon$ has no pole on $B$; and the construction takes its interpolation error below $\min_i t_i$, so on each input ball $F_\varepsilon=g_i+h$ with $\sup|h|_p<t_i$. The following lemma then upgrades the inclusion to the equality, and it is also what a bound on $\varepsilon$ buys when bijectivity is wanted.

\begin{lemma}[Image-ball stability under small perturbation; {cf.\ \cite[Lemma~3.5]{rogelio2023gluing}}]
\label{lem:image_ball_perturbation}
Let $r\in|\C_p^\times|$ and let $g,h \in \C_p(z)$ have no poles on the closed ball $\BCP_r(a)$, with $g(\BCP_r(a)) = \BCP_t(b)$. If $q:=\sup_{z \in \BCP_r(a)} |h(z)|_p < t$, then $(g+h)(\BCP_r(a)) = \BCP_t(b)$. Moreover, if $g$ is bijective on $\BCP_r(a)$, then so is $g+h$.
\end{lemma}

\begin{proof}
Being pole-free on $\BCP_r(a)$, both maps are given there by convergent power series in $z-a$, say $g(z)=g(a)+\sum_{k\ge1}a_k(z-a)^k$ and $h(z)=\sum_{k\ge0}b_k(z-a)^k$. On a rational closed ball the sup norm is the Gauss norm, so $\sup_{\BCP_r(a)}|g-g(a)|_p=\max_{k\ge1}|a_k|_pr^k$, and since that supremum is the radius of the image ball,
\[
t=\max_{k\ge1}|a_k|_p\,r^k,\qquad\text{while}\qquad |b_k|_p\,r^k\le q<t\ \ (k\ge0).
\]
Fix $k_0\ge1$ attaining the first maximum. Then $|b_{k_0}|_pr^{k_0}<|a_{k_0}|_pr^{k_0}$, so the strong triangle inequality is an equality there and $|a_{k_0}+b_{k_0}|_pr^{k_0}=t$, while $|a_k+b_k|_pr^k\le\max\{|a_k|_pr^k,|b_k|_pr^k\}\le t$ for every $k$. Hence $\max_{k\ge1}|a_k+b_k|_pr^k=t$, and the image of $\BCP_r(a)$ under $g+h$ is the closed ball of radius $t$ about $(g+h)(a)$. Its center lies within $t$ of $b$: the hypothesis gives $g(a)\in\BCP_t(b)$, so $|(g+h)(a)-b|_p=|g(a)-b+h(a)|_p\le\max\{|g(a)-b|_p,|h(a)|_p\}\le\max\{t,q\}=t$. Two closed balls of radius $t$ whose centers are within $t$ of each other coincide, so that ball is $\BCP_t(b)$.

If $g$ is bijective on $\BCP_r(a)$ its Weierstrass degree there is $1$, again by the multiplicity argument in the proof of Lemma~\ref{lem:closed_ball_bridge}(b), which by \cite[Definition~3.8]{benedetto2019dynamics} means $|a_1|_pr=t$ and $|a_k|_pr^k<t$ for $k\ge2$. The same two estimates then give $|a_1+b_1|_pr=t$ and $|a_k+b_k|_pr^k<t$ for $k\ge2$, so $g+h$ has Weierstrass degree $1$ as well and is bijective onto $\BCP_t(b)$ by \cite[Theorem~3.15]{benedetto2019dynamics}, as claimed.
\end{proof}

\noindent This is the workhorse of Section~\ref{sec:hierarchical_framework}: a perturbation whose sup norm stays strictly below the image radius changes neither the image ball nor bijectivity. The statement in~\cite{rogelio2023gluing} is for open disks, and the proof above is the closed-ball version, which is what the constructions of Section~\ref{sec:hierarchical_framework} use. Attaining the radius is what the argument needs, and a point of the image at distance exactly $t$ from $b$ would not suffice on its own, since a proper sub-ball of $\BCP_t(b)$ can meet that sphere.

%% file: 04_hierarchical_framework.tex
\section{\texorpdfstring{Hierarchical $p$-Adic Ball Dynamics}{Hierarchical p-Adic Ball Dynamics}}
\label{sec:hierarchical_framework}

Throughout this section we work with the embedding $\mathcal{C} \hookrightarrow \mathbb{Z}_p$ of Section~\ref{sec:setup} (Eq.~\eqref{eq:ball_association}), writing $\BCP(a) := \BCP_{1/p^N}(m_a)$ for the analytic ball of a configuration and $\mathcal{U} := \bigsqcup_{a} \BCP(a)$ for their union, an affinoid domain fixed by the embedding and independent of $f$. We model the discrete dynamics $f \colon \mathcal{C} \to \mathcal{C}$ by rational approximations $F_0, F_1, \ldots, F_N \in \mathbb{C}_p(z)$, one per scale, with Theorem~\ref{T1} providing the analytic backbone. Section~\ref{sec:toy_example} runs the construction in a small case ($p=2$, $N=4$), with the image-ball data in the Supplementary Material.

\subsection{Recap: rational interpreters of finite discrete dynamics}
\label{subsec:recap_interpreters}

Before constructing the multi-scale dynamics $F_0, \ldots, F_N$, we briefly recall the construction that lets us realize any finite map $f \colon \mathcal{C} \to \mathcal{C}$ as the trace of a $p$-adic rational function. This subsection is self-contained. The reader familiar with the gluing construction of~\cite{rogelio2023gluing} may skip to Section~\ref{subsec:rational_approximations}, where the radii and centers used throughout are fixed.

\paragraph{Interpreting discrete dynamics.}
The \emph{projection} $\mathrm{proj}_{\mathcal{U}} \colon \mathcal{U} \to \mathcal{C}$ sends each $x \in \mathcal{U}$ to the unique $a \in \mathcal{C}$ such that $x \in \BCP(a)$.

\begin{definition}
\label{def:interprets}
    Let $\phi \in \C_p(z)$ be a rational function. We say that $\phi$ \emph{interprets the dynamics of $f$} if
\[
\phi(\BCP(a)) \subseteq \BCP(f(a)) \quad \text{for all } a \in \mathcal{C}.
\]
Equivalently, $\phi$ maps $\mathcal{U}$ to itself and the following diagram commutes:
\[
\begin{CD}
\mathcal{U}   @>{\phi}>>  \mathcal{U}\\
@V{\mathrm{proj}_{\mathcal{U}}}VV   @VV{\mathrm{proj}_{\mathcal{U}}}V\\
\mathcal{C}   @>{f}>>     \mathcal{C}
\end{CD}
\]
\end{definition}

\begin{remark}[The finite system is a factor of the rational one]
\label{rem:semiconjugacy}
Since each $\BCP(a)$ is clopen, $\mathrm{proj}_{\mathcal{U}}$ is locally constant, hence a \emph{semi-conjugacy} from $(\mathcal{U}, \phi)$ to $(\mathcal{C}, f)$ that preserves the projected transition structure, including periodic orbits on the finite quotient. In this sense $\phi$ is a rational extension of the finite dynamics $f$, recovered as its shadow under $\mathrm{proj}_{\mathcal{U}}$.
\end{remark}

\paragraph{Existence of rational interpreters.}
Existence of a rational interpreter (Definition~\ref{def:interprets}) is a direct consequence of the $\varepsilon$-approximation theorem (Theorem~\ref{T1}, Section~\ref{sec:setup}).

\begin{theorem}[Existence of rational interpreters. Cf.~Theorem~4.2 of \cite{rogelio2023gluing}]
\label{thm:recap_interpreter_existence}
For every $f \colon \mathcal{C} \to \mathcal{C}$ and every embedding $\iota$ as above, there exists $\phi \in \mathbb{C}_p(z)$ that interprets $f$.
\end{theorem}

\begin{proof}
Apply the $\varepsilon$-approximation theorem (Theorem~\ref{T1}) to the local affine maps $g_a(z)=A_a(z-m_a)+m_{f(a)}$ with $|A_a|_p=1$, one for each $a\in\mathcal{C}$, on the pairwise disjoint rational balls $\{\BCP(a)\}_{a\in\mathcal{C}}$. Each $g_a$ is an isometry carrying $\BCP(a)=\BCP_{1/p^N}(m_a)$ onto $\BCP_{1/p^N}(m_{f(a)})=\BCP(f(a))$, a ball of $\C_p$ that meets $\mathbb{Z}_p$ but is not contained in it. Hypotheses (H1)--(H2) of Theorem~\ref{T1} therefore hold, and so does (H3): for $z$ in any source ball, $|g_a(z)|_p\le\max\{|A_a|_p|z-m_a|_p,\,|m_{f(a)}|_p\}\le1$, since $|A_a|_p=1$ and both $z$ and $m_a$ lie in $\BCP_1(0)$, so $g_a$ carries the whole of $\bigsqcup_b\BCP(b)$ into $\BCP_1(0)$. Conclusion~\ref{T1:image} of Theorem~\ref{T1} is an equality for every $\varepsilon>0$, so the resulting $F_\varepsilon$ satisfies $F_\varepsilon(\BCP(a))=\BCP(f(a))$ for every $a$ with no smallness assumption, hence interprets $f$ in the sense of Definition~\ref{def:interprets}. The explicit partition-of-unity form of $F_\varepsilon$ is the one recalled after Theorem~\ref{T1}.
\end{proof}

\noindent The interpreter is not unique, and Section~\ref{subsec:rational_approximations} below uses this freedom to build a hierarchy of approximations $F_0, \ldots, F_N$ at each resolution level. Interpreters are studied on their own in~\cite{perez2026interpreters}, where existence is obtained by rigid-analytic Runge approximation and the contracting, indifferent and expanding regimes are read off a linear-dominance condition. That work fixes a single interpreter. What is built here is a family indexed by resolution, one approximant per scale, with image radii prescribed by the finite data (Theorem~\ref{thm:existence_approximations}).

\paragraph{Downstream requirements.}
The remainder of the paper requires only three ingredients from this construction:

\begin{enumerate}
\item that $\phi$ exists (Theorem~\ref{thm:recap_interpreter_existence});
\item that the canonical radii $t_{n,m}$, equivalently the common-prefix exponents $M_{n,m}$, are determined by $(f,\iota)$ alone and bound the analytic image radii of any interpreter (Remark~\ref{rem:tnm_computation}, below); and
\item the attracting/repelling/indifferent fixed-point trichotomy for non-Archimedean rational dynamics on a closed ball (Lemma~\ref{lem:closed_ball_bridge}, Section~\ref{sec:setup}).
\end{enumerate}

The interpreter enters through its existence, and Section~\ref{sec:stability_measure} computes the scores from $(f,\iota)$ directly.

\begin{definition}[Exact interpreter]
\label{def:exact_interpreter}
An interpreter $\phi$ of $f$ is \emph{exact} if the inclusion of Definition~\ref{def:interprets} is an equality:
\[
\phi(\BCP(a)) = \BCP(f(a)) \quad \text{for all } a \in \mathcal{C}.
\]
\end{definition}

\begin{proposition}[Existence of an exact interpreter]
\label{prop:exact_existence}
For every $f \colon \mathcal{C} \to \mathcal{C}$ and every embedding $\iota$ there exists an exact interpreter.
\end{proposition}

\begin{proof}
Note that the interpreter produced in the proof of Theorem~\ref{thm:recap_interpreter_existence} is already exact. Conclusion~\ref{T1:image} of Theorem~\ref{T1} returns the equality $F_\varepsilon(\BCP(a)) = \BCP(f(a))$ for every $a \in \mathcal{C}$, which is Definition~\ref{def:exact_interpreter}. Moreover, choosing $\varepsilon < 1/p^N$ makes $F_\varepsilon - g_a$ of sup norm strictly below the image radius on each source ball, so Lemma~\ref{lem:image_ball_perturbation} gives bijectivity of $F_\varepsilon$ on each $\BCP(a)$, since each local isometry $g_a$ is bijective. Non-injectivity of $f$ is no obstruction: distinct source balls may share a common image ball.
\end{proof}

\subsection{Hierarchical Rational Approximations}
\label{subsec:rational_approximations}

Let $\phi \in \C_p(z)$ interpret $f$. We construct approximations $F_0, \ldots, F_N \in \C_p(z)$ that agree with $\phi$ on each level-$n$ ball to within the prescribed image radius $t_{n,m}$ of that ball (Proposition~\ref{Prop:AproxFnPhi}). In particular $|F_N(z) - \phi(z)|_p \leq 1/p^N$ for all $z \in \mathbb{Z}_p$. The $F_n$ yield the ball-level classification and the local fixed-point structure (Section~\ref{subsec:quasi_dynamics}).

For $n \geq 0$, $\mathbb{Z}_p = \bigsqcup_{m=0}^{p^n-1} \BZP_{1/p^n}(m)$. Write $B_{n,m} := \BZP_{1/p^n}(m) \subset \mathbb{Z}_p$, so $B_{N,m} = \BZP(m) = \BZP_{1/p^N}(m)$. Each $a \in \mathcal{C}$ has center $m_a = \sum_{i=0}^{N-1} a_i p^i$. Thus $\BCP(a) = \BCP_{1/p^N}(m_a)$ and $\iota(a) = \BZP(a) = \BZP_{1/p^N}(m_a)$. For any $z \in \mathbb{Z}_p$ and $n \le N$, the \emph{$n$-digit truncation} is $\operatorname{trunc}_n(z) := z \bmod p^n \in \{0,\dots,p^n-1\}$. For a configuration $a \in \mathcal{C}$, we write $\operatorname{trunc}_n(a)$ as shorthand for $\operatorname{trunc}_n(m_a)$, where $m_a \in \mathbb{Z}$ is the integer center of $\BCP(a)$. Then $z \in B_{n,m}$ if and only if $\operatorname{trunc}_n(z) = m$.

For each ball $B_{n,m}$, define first the canonical common-prefix exponent
\begin{equation}
\label{eq:Mnm_definition}
\begin{aligned}
M_{n,m}:=\max\Bigl\{k\in\{0,\ldots,N\}:{}&\ m_{f(a)}\equiv m_{f(b)} \pmod{p^k}\\
&\text{for all } a,b\in\mathcal{C}
\text{ with } \operatorname{trunc}_n(a)=\operatorname{trunc}_n(b)=m\Bigr\}.
\end{aligned}
\end{equation}
The set is nonempty because $k=0$ always works. Thus $M_{n,m}$ is the number of initial ordered digits common to all image configurations of points lying in the level-$n$ ball indexed by $m$. Set
\begin{equation}
\label{eq:tnm_definition}
t_{n,m}:=p^{-M_{n,m}}.
\end{equation}
This is a finite combinatorial quantity. It bounds the diameter of $\phi(B_{n,m})\subset\mathbb{C}_p$ for every interpreter $\phi$, and it is attained by the exact interpreters constructed in Proposition~\ref{prop:exact_existence} (Proposition~\ref{prop:radius_attained}). Exactness of the analytic image ball alone does not force the attainment on the observable points. The same number will be used below as the prescribed image radius for the $n$-th approximants.

\begin{remark}[Computation of $t_{n,m}$ from discrete data]
\label{rem:tnm_computation}
The values $t_{n,m}$ can be computed directly from the transition table of $f$:

\begin{enumerate}
    \item Enumerate the configurations $a\in\mathcal{C}$ with $\operatorname{trunc}_n(a)=m$.
    
    \item For their images, compute the largest $k$ such that all integer centers $m_{f(a)}$ are congruent modulo $p^k$.
    
    \item Set $M_{n,m}=k$ and $t_{n,m}=p^{-k}$.
\end{enumerate}

Equivalently,
\[
t_{n,m}=\max\Bigl(\{p^{-N}\}\cup\{|m_{f(a)}-m_{f(b)}|_p:
\operatorname{trunc}_n(a)=\operatorname{trunc}_n(b)=m\}\Bigr).
\]
For any interpreter $\phi$ and any configurations $a,b$ in the same level-$n$ ball,
	    \begin{equation}
	    \label{eq:distance_from_images}
	    \begin{cases}
        |\phi(x) - \phi(y)|_p \leq 1/p^N, & \text{if } f(a) = f(b), \\
        |\phi(x) - \phi(y)|_p = 1/p^j, & \text{if } f(a) \neq f(b),
	    \end{cases}
	    \end{equation}
where $x\in\BCP(a)$, $y\in\BCP(b)$, and $j$ is the first digit at which $m_{f(a)}$ and $m_{f(b)}$ differ. Thus $\diam(\phi(B_{n,m}))\le t_{n,m}$, with equality for the exact interpreters constructed above.

The entire analysis (classification of Definition~\ref{def:quasi_dynamics} in Section~\ref{subsec:quasi_dynamics} below, and the scores of Section~\ref{sec:stability_measure}) therefore depends only on $f$, not on $\phi$.
\end{remark}

\begin{lemma}[Monotonicity of $t_{n,m}$]
\label{lem:tnm_monotonicity}
If $\BZP_{1/p^{n+1}}(m') \subseteq \BZP_{1/p^n}(m)$, then $t_{n+1,m'} \leq t_{n,m}$.
\end{lemma}

\begin{proof}
The configurations in the sub-ball $\BZP_{1/p^{n+1}}(m')$ form a subset of the configurations in $\BZP_{1/p^n}(m)$. Note that taking the common prefix length over fewer image configurations can only increase $M_{n,m}$, so $M_{n+1,m'}\ge M_{n,m}$ and therefore $t_{n+1,m'}=p^{-M_{n+1,m'}}\le p^{-M_{n,m}}=t_{n,m}$, as required.
\end{proof}

\medskip

The radius $t_{n,m}$ prescribes how large the image of a level-$n$ ball must be. Its center is prescribed too, and by the same finite data.

\begin{definition}[The canonical target ball]
\label{def:canonical_target}
Fix $0\le n\le N$ and $m\in\{0,\dots,p^n-1\}$. The level-$n$ ball indexed by $m$ is nonempty, since the residue class $m$ has exactly $p^{N-n}$ configurations extending it, and by~\eqref{eq:Mnm_definition} the image centers $m_{f(a)}$ of all of them agree modulo $p^{M_{n,m}}$. The residue class
\begin{equation}
\label{eq:canonical_center}
\bar\beta_{n,m}\;:=\;m_{f(a)}\bmod p^{M_{n,m}}\;\in\;\mathbb{Z}/p^{M_{n,m}}\mathbb{Z},
\qquad \text{any } a \text{ with } \operatorname{trunc}_n(a)=m,
\end{equation}
is therefore well defined and determined by $(f,\iota)$ alone. Let $\beta_{n,m}\in\mathbb{Z}_p$ be any lift of $\bar\beta_{n,m}$. The \emph{canonical target ball} at $(n,m)$ is the analytic ball $\BCP_{t_{n,m}}(\beta_{n,m})\subset\C_p$, whose observable trace is the residue class $\beta_{n,m}+p^{M_{n,m}}\mathbb{Z}_p$. Neither depends on the lift, since every point of a non-Archimedean ball is a center: the canonical object is the residue class~\eqref{eq:canonical_center}, and the choice of $\beta_{n,m}$ is notation.
\end{definition}

\begin{remark}[The canonical target holds every image configuration of the source ball]
\label{rem:target_contains_images}
For every $a$ with $\operatorname{trunc}_n(a)=m$ one has $|m_{f(a)}-\beta_{n,m}|_p\le p^{-M_{n,m}}=t_{n,m}$, so
\[
\bigl\{\,m_{f(a)}\ :\ \operatorname{trunc}_n(a)=m\,\bigr\}\ \subseteq\ \BCP_{t_{n,m}}(\beta_{n,m}).
\]
This restates the definition of $M_{n,m}$. Among the balls of the hierarchy, $\BCP_{t_{n,m}}(\beta_{n,m})$ is the smallest that contains all the image configurations of the source ball, since the maximality in~\eqref{eq:Mnm_definition} rules out every smaller radius.
\end{remark}

\medskip

\noindent\emph{Standing assumption for the propositions that follow.} From this point through the end of Section~\ref{sec:hierarchical_framework}, $\phi$ denotes an exact interpreter (Definition~\ref{def:exact_interpreter}), which exists for every $f$ by Proposition~\ref{prop:exact_existence}, so that $\phi(\BCP(a)) = \BCP(f(a))$.

\begin{definition}[$n$-th approximation]
\label{def:nth_approximation}
An \emph{$n$-th approximation of $f$}, also called an \emph{$n$-th approximant}, is a rational function $F_n \in \C_p(z)$ whose image of every level-$n$ ball is the canonical target ball of Definition~\ref{def:canonical_target}:
\[
F_n(\BCP_{1/p^n}(m)) = \BCP_{t_{n,m}}(\beta_{n,m}) \quad\text{as balls in } \C_p,
\qquad m \in \{0, \ldots, p^n-1\}.
\]
\end{definition}

\noindent Radius and center are thus both fixed by $(f,\iota)$ before any interpreter is chosen: an $n$-th approximant must send each source ball to the place the finite data prescribe. The integral center also guarantees a nonempty observable trace.

\noindent By Remark~\ref{rem:target_contains_images}, $\phi(m)$ lies in the canonical target ball for the exact interpreter $\phi$: taking $a$ to be the configuration with $\operatorname{trunc}_n(a)=m$ and all later digits zero gives $\phi(m)\in\BCP(f(a))\subseteq\BCP_{t_{n,m}}(\beta_{n,m})$. So the same ball may be written about $\phi(m)$ whenever that is convenient, as it is in Proposition~\ref{Prop:AproxFnPhi}.

\noindent The standing assumption isolates the sharp statements below from the generic case of Remark~\ref{rem:exact_vs_generic}, where the equality weakens to the inclusion $\subseteq$. For generic interpreters merely satisfying Definition~\ref{def:interprets}, the propositions go through with $\subseteq$ in place of $=$ in their hypotheses, since the proofs use only the upper bounds $|F_n(z) - \phi(m)|_p \le 1/p^M$.

For a generic interpreter the analytic image is only \emph{contained} in $\BCP_{t_{n,m}}(\beta_{n,m})$ and can be strictly smaller, e.g.\ when $f$ is constant on $B_{n,m}$ (where the finite transition data give $t_{n,m} = 1/p^N$ while the actual image may be a single point). The combinatorial $t_{n,m}$ and the scores $\mu_E, \mu_A, \mu_I$ are unaffected (Proposition~\ref{thm:intrinsic_mu} and Remark~\ref{rem:exact_vs_generic}, both in Section~\ref{sec:stability_measure} below).

\begin{theorem}[Existence of hierarchical approximations]
\label{thm:existence_approximations}
Let $f:\mathcal{C}\to\mathcal{C}$ be a discrete dynamical system. For each resolution level $n \in \{0, \ldots, N\}$, there exists a rational function $F_n \in \C_p(z)$ satisfying Definition~\ref{def:nth_approximation}. That is, $F_n(\BCP_{1/p^n}(m)) = \BCP_{t_{n,m}}(\beta_{n,m})$ is a ball of radius $t_{n,m}$ for every $m$.
\end{theorem}

\noindent\textit{Strategy of the proof.} For each $m$ we exhibit a local rational piece $\hat g_m$ satisfying three properties (P1)--(P3), which together yield the hypotheses (H1)--(H3) of Theorem~\ref{T1}. The bare affine suffices when the image radius does not exceed the source radius ($t_{n,m} \le 1/p^n$, the non-expanding case). When $t_{n,m} > 1/p^n$ (expanding case) we localize it with a rational gluing factor of degree~$2$ whose pole sphere has half-integer radius, hence misses $\mathbb{Z}_p$. Theorem~\ref{T1} then produces a global $F_n$ within $\varepsilon$ of each $\hat g_m$ on its own ball, and~Lemma~\ref{lem:image_ball_perturbation} closes the image-ball equality and injectivity.

\begin{proof}
Fix $0 \le n \le N$. For $n=0$ the partition is the single ball $\mathbb{Z}_p$ and $F_0(z) := A_0\,z + \beta_{0,0}$ with $|A_0|_p = t_{0,0}$ is one affine map onto $\BCP_{t_{0,0}}(\beta_{0,0})$. Assume $n \ge 1$ below. For each $m \in \{0, \ldots, p^n-1\}$ take as the lift of the canonical class~\eqref{eq:canonical_center} the integer center $\beta_{n,m} := m_{f(a)}$, for any configuration $a$ with $\operatorname{trunc}_n(a)=m$, so that $\beta_{n,m}\in\mathbb{Z}_p$ and $|\beta_{n,m}|_p \le 1$. Write $\Lambda := t_{n,m}\,p^n$. The balls $\{\BCP_{1/p^n}(m)\}_{m=0}^{p^n-1}$ are pairwise disjoint (Section~\ref{subsec:embedding}).

We exhibit a family $\{\hat{g}_m\}_{m=0}^{p^n-1} \subset \mathbb{C}_p(z)$ satisfying:
\begin{description}
\item[(P1)] $\hat{g}_m$ maps $\BCP_{1/p^n}(m)$ bijectively onto $\BCP_{t_{n,m}}(\beta_{n,m})$;
\item[(P2)] for every $m' \ne m$, $\hat{g}_m(\BCP_{1/p^n}(m')) \subset \BCP_1(0)$;
\item[(P3)] $\hat{g}_m$ has no pole on $\mathbb{Z}_p$.
\end{description}
These imply the hypotheses of Theorem~\ref{T1}, as collected at the end of the proof: each $\hat{g}_m$ is rational on the affinoid $\mathcal{U} \subset \BCP_1(0)$, with image in the unit ball.

\emph{Case A: $\Lambda \le 1$ (non-expanding ball).} Take the bare affine $\hat{g}_m(z) := A_m(z - m) + \beta_{n,m}$ with $A_m := p^{\,M_{n,m}-n}\in\mathbb{Q}_p$, so that $|A_m|_p = \Lambda$. (P3) is immediate. For $z \in \BCP_{1/p^n}(m)$, $|\hat{g}_m(z) - \beta_{n,m}|_p = \Lambda\,|z-m|_p$ attains $\Lambda \cdot p^{-n} = t_{n,m}$ on the boundary and is strictly smaller in the interior. The affine map is bijective, giving (P1). For $z$ in any other source ball, $|\hat{g}_m(z)|_p \le \max\{\Lambda\,|z-m|_p,\,|\beta_{n,m}|_p\} \le 1$ since $\Lambda \le 1$, $|z-m|_p \le 1$, and $|\beta_{n,m}|_p \le 1$, giving (P2).

\emph{Case B: $t_{n,m} > 1/p^n$ (expanding ball).} The bare affine violates~(P2): on a distinct ball $|A_m(z-m)|_p$ reaches $\Lambda > 1$. Localize with a rational gluing factor of degree~$2$:
\begin{equation}
\label{eq:ghat_expander}
\begin{gathered}
\hat{g}_m(z) \;:=\; \frac{A_m(z - m)}{1 - \bigl((z - m)/c_m\bigr)^{2}} + \beta_{n,m},\\[3pt]
A_m := p^{\,M_{n,m}-n}\in\mathbb{Q}_p,\qquad |A_m|_p = \Lambda,\qquad |c_m|_p = p^{\,1/2 - n}.
\end{gathered}
\end{equation}
Note that the formula uses only $c_m^{\,2}$. Choose the representative $c_m^{\,2} := p^{\,2n - 1} \in \mathbb{Q}_p$, of integer valuation $2n - 1$ (equivalently $|c_m^{\,2}|_p = p^{\,1 - 2n}$, consistent with $|c_m|_p = p^{\,1/2 - n}$). Then $\hat{g}_m \in \mathbb{Q}_p(z)$, although its poles, the two square roots of $c_m^{\,2}$, lie in the ramified extension $\mathbb{Q}_p(\sqrt{p}\,)$. The poles sit on the sphere $|z - m|_p = p^{\,1/2 - n}$, of half-integer radius. For every $\zeta \in \mathbb{Z}_p$ the value $|\zeta - m|_p$ is either $0$, when $\zeta = m$, or an integer power $p^{-k}$ with $k \ge 0$. Neither is $p^{\,1/2-n}$, so no point of $\mathbb{Z}_p$ lies on that sphere of half-integer exponent, giving (P3).

For $z \in \BCP_{1/p^n}(m)$, $|z - m|_p \le p^{-n}$ and $|c_m|_p = p^{\,1/2 - n}$ give $|(z-m)/c_m|_p^{\,2} \le p^{-1} < 1$, so by the strong triangle inequality $|1 - ((z-m)/c_m)^2|_p = 1$ and
\[
\hat{g}_m(z) - \beta_{n,m} \;=\; A_m(z - m)\,\bigl(1 - u\bigr)^{-1},
\qquad u := \bigl((z-m)/c_m\bigr)^{2},\ |u|_p \le 1/p.
\]
The geometric expansion $\bigl(1 - u\bigr)^{-1} = 1 + u + u^2 + \cdots$ converges in $\mathbb{C}_p$ since $|u|_p < 1$, so $\hat{g}_m(z) - \beta_{n,m} = A_m(z - m) + r(z)$ with $\sup_{\BCP_{1/p^n}(m)} |r|_p \le |A_m(z-m)|_p \cdot |u|_p \le t_{n,m}/p < t_{n,m}$. The bare affine $z \mapsto A_m(z-m) + \beta_{n,m}$ is bijective from $\BCP_{1/p^n}(m)$ onto $\BCP_{t_{n,m}}(\beta_{n,m})$, and the perturbation $r$ has no pole on that ball, its poles being the two square roots of $c_m^{\,2}$ on the sphere of radius $p^{\,1/2-n}>p^{-n}$, and sup norm strictly below the image radius. By~Lemma~\ref{lem:image_ball_perturbation}, $\hat{g}_m$ maps $\BCP_{1/p^n}(m)$ onto the same ball $\BCP_{t_{n,m}}(\beta_{n,m})$ and is injective there, giving (P1).

Observe that for $z$ in any other source ball, $|z - m|_p \ge p^{\,1-n}$, so $|(z-m)/c_m|_p^{\,2} \ge p > 1$ and $|1 - ((z-m)/c_m)^2|_p = |(z-m)/c_m|_p^{\,2}$. Hence
\begin{equation}
\label{eq:Fn_split}
\begin{split}
\Bigl|\hat{g}_m(z) - \beta_{n,m}\Bigr|_p
&\;=\; \frac{|A_m|_p\,|c_m|_p^{\,2}}{|z - m|_p}
\;=\; \frac{(t_{n,m}\,p^n) \cdot p^{\,1 - 2n}}{|z - m|_p}
\;=\; \frac{t_{n,m}\, p^{\,1-n}}{|z - m|_p}\\
&\;\le\; \frac{t_{n,m}\, p^{\,1-n}}{p^{\,1-n}}
\;=\; t_{n,m} \;\le\; 1,
\end{split}
\end{equation}
so $|\hat{g}_m(z)|_p \le \max\{|\hat{g}_m - \beta_{n,m}|_p,\,|\beta_{n,m}|_p\} \le 1$, giving (P2).

The family $\{\hat{g}_m\}_{m=0}^{p^n-1}$ now satisfies the hypotheses of Theorem~\ref{T1}. The source balls are pairwise disjoint and their common radius $p^{-n}$ lies in $p^{\mathbb{Q}}$, so they are rational, giving (H1). Property (P1) is the exact prescribed image required by (H2). Every image lies in $\BCP_1(0)$: on the own ball by (P1), since $|\beta_{n,m}|_p \le 1$ and $t_{n,m} \le 1$, and on the other balls by (P2), giving (H3); property (P3) makes each $\hat g_m$ defined on all of $\mathbb{Z}_p$. By Theorem~\ref{T1}, for every $\varepsilon > 0$ there exists $F_n^{(\varepsilon)} \in \mathbb{C}_p(z)$ with $\sup_{z \in \BCP_{1/p^n}(m)} |F_n^{(\varepsilon)}(z) - \hat{g}_m(z)|_p < \varepsilon$ for every $m$, and $F_n^{(\varepsilon)}(\mathbb{Z}_p) \subset \BCP_1(0)$. Choose any $\varepsilon < \min_{m} t_{n,m}$ and set $F_n := F_n^{(\varepsilon)}$. On each $\BCP_{1/p^n}(m)$,
\begin{equation}
\label{eq:Fn_sup_bound}
\sup_{z \in \BCP_{1/p^n}(m)} |F_n(z) - \hat{g}_m(z)|_p \;\le\; \varepsilon \;<\; t_{n,m},
\end{equation}
the image radius of $\hat{g}_m$, which is bijective onto $\BCP_{t_{n,m}}(\beta_{n,m})$. Theorem~\ref{T1} gives its inequality pointwise, so the supremum is bounded by $\varepsilon$, which suffices because $\varepsilon<t_{n,m}$. Neither map has a pole on $\BCP_{1/p^n}(m)$: those of $\hat g_m$ lie on a sphere of radius $p^{\,1/2-n}>p^{-n}$ in the expanding case and do not exist in the other, while the image equality of Theorem~\ref{T1} makes $F_n(\BCP_{1/p^n}(m))$ a bounded ball of $\C_p$, and a pole of $F_n$ in the source ball would place $\infty$ in its image. By~Lemma~\ref{lem:image_ball_perturbation}, $F_n$ maps $\BCP_{1/p^n}(m)$ onto the same image ball $\BCP_{t_{n,m}}(\beta_{n,m})$ and is injective there. This completes the proof.
\end{proof}

\begin{remark}
\label{rem:existence_uniformity}
The image-ball equality of Definition~\ref{def:nth_approximation} and the injectivity of $F_n$ on each source ball hold for any $\varepsilon < \min_m t_{n,m}$, by~Lemma~\ref{lem:image_ball_perturbation}, and not merely in the limit $\varepsilon\to0$. That injectivity is the degree-one input for the degree-one hypothesis of Lemma~\ref{lem:closed_ball_bridge}(b).
\end{remark}

This yields a finite sequence $F_0, F_1, \ldots, F_N$ satisfying Definition~\ref{def:nth_approximation}. The approximation bound $|\phi(z) - F_n(z)|_p \le t_{n,m}$ for $z \in B_{n,m}$ is given by Proposition~\ref{Prop:AproxFnPhi} below. Non-uniqueness and invariance of classification are in Proposition~\ref{prop:perturbation} and Section~\ref{sec:stability_measure}.

\subsection{Properties of the Approximation Sequence}
\label{subsec:approximation_properties}

Two properties of $\{F_n\}$ are established next, under the standing assumption: $\phi$ is exact, so Definition~\ref{def:nth_approximation} may be read as $F_n(\BCP_{1/p^n}(m)) = \BCP_{t_{n,m}}(\phi(m))$, since $\phi(m)$ lies in the canonical target ball and every point of a ball is a center. The first says that $F_n$ approximates $\phi$ to the accuracy of the level-$n$ image radius, the second that the approximants are stable under small perturbations.

\begin{proposition}
\label{Prop:AproxFnPhi}
    Let $F_n$ be an $n$-th approximation and let $m \in \{0,\ldots,p^n-1\}$, with $M_{n,m}$ the canonical exponent ($t_{n,m} = 1/p^{M_{n,m}}$). Then:
    \begin{equation}
    \label{eq:Fn_phi_approximation}
    |\phi(z)-F_n(z)|_p\leq\frac{1}{p^{M_{n,m}}}\quad\text{for all }z\in \BZP_{1/p^n}(m).
    \end{equation}
\end{proposition}

\begin{proof}
Both $F_n(z)$ and $\phi(z)$ lie within $1/p^{M_{n,m}}$ of $\phi(m)$: the first because $F_n$ maps the ball onto $\BCP_{1/p^{M_{n,m}}}(\phi(m))$, the second because every micro-ball $\BCP(a)$ inside $\BZP_{1/p^n}(m)$ is carried into $\BCP(f(a))$, which sits in that same target ball by Remark~\ref{rem:target_contains_images}. The strong triangle inequality then gives the bound, as required.
\end{proof}

\begin{proposition}
\label{prop:perturbation}
    Let $F_n\in\C_p(z)$ be an $n$-th approximation. Put
    \[
	    \hat{M}=\max_{0\leq m\leq p^n-1} M_{n,m},
	    \]
    so $p^{-\hat M}=\min_m p^{-M_{n,m}}$ is the smallest image radius among the level-$n$ balls. Let $h\in\C_p(z)$ be a rational function such that:
    \begin{equation}
    \label{eq:perturbation_bound}
    \sup_{z\in\bigsqcup_{m=0}^{p^n-1}\BCP_{1/p^n}(m)}|h(z)|_p<\frac{1}{p^{\hat{M}}}.
    \end{equation}
    Then $F'_n:=F_n+h\in\C_p(z)$ is also an $n$-th approximation.
\end{proposition}

\begin{proof}
Fix $m$. By~\eqref{eq:perturbation_bound} the sup norm of $h$ on $\BCP_{1/p^n}(m)$ is below $1/p^{\hat M}\le 1/p^{M_{n,m}}$, the image radius there, so Lemma~\ref{lem:image_ball_perturbation} leaves the image ball unchanged.
\end{proof}

\subsection{Classification of Ball-Level Dynamics}
\label{subsec:quasi_dynamics}

\begin{definition}
\label{def:quasi_dynamics}
    Let $f:\mathcal{C}\to\mathcal{C}$ be a dynamical system, let $\mathcal{U} = \bigcup_{a\in\mathcal{C}}\BCP(a)$ be the affinoid domain determined by the embedding (independently of $f$), and let $t_{n,m}$ be as in~\eqref{eq:tnm_definition}. For $m\in\{0,\ldots,p^n-1\}$, the \emph{coarse multiplier} of the ball $\BCP_{1/p^n}(m)$ is the ratio of image radius to domain radius,
    \[
        \Lambda_{n,m} \;:=\; \frac{t_{n,m}}{1/p^n} \;=\; p^n\, t_{n,m} \;=\; p^{\,n-M_{n,m}},
    \]
    the non-Archimedean ball-level analogue of the classical multiplier $|\phi'(x_0)|_p$ at a fixed point (the comparison is made precise in Section~\ref{sec:stability_measure}). The ratio records the radius of the image ball relative to that of its source, under a map that may move the ball to a different branch of the hierarchy. It is not the local degree of a fixed point of Type II, which is a positive integer and therefore admits no attracting case~\cite[Warning~8.5]{benedetto2019dynamics}. The three cases below classify the ball by that radius comparison alone, and say nothing about where the image sits: the nesting that turns them into trapping, covering or isometric self-dynamics is the subject of Remarks~\ref{rem:quasi_Zp} and~\ref{rem:fixed_points}. We say that the dynamics of $f$ in $\BCP_{1/p^n}(m)$ is
    \begin{itemize}
        \item \emph{$n$-contracting} if $\Lambda_{n,m} < 1$ (equivalently $t_{n,m} < 1/p^n$);
        \item \emph{$n$-expanding} if $\Lambda_{n,m} > 1$ (equivalently $t_{n,m} > 1/p^n$);
        \item \emph{$n$-isometric} if $\Lambda_{n,m} = 1$ (equivalently $t_{n,m} = 1/p^n$).
    \end{itemize}
\end{definition}

The value $t_{n,m}$ captures the dynamical spread within $B_{1/p^n}(m)$: if all configurations in the ball map to the same configuration, $t_{n,m} = 1/p^N$ (minimal spread). If they map to different configurations, $t_{n,m} > 1/p^N$. The classification depends on this spread, and refinement from level $n$ to level $n+1$ can reveal sub-balls with different types. Figure~\ref{fig:grn_to_padic_pipeline} illustrates the passage from a finite Boolean transition table to the induced $p$-adic ball partition and its ball-level labels.

\begin{definition}[The $A/E/I$ word of a configuration]
\label{def:aei_word}
Fix an ordering of the coordinates, written $\pi$ and formalized in Section~\ref{sec:stability_measure}, and a configuration $a\in\mathcal{C}$. For each level $1\le n\le N-1$ the configuration lies in exactly one ball $B_{n,m}$ of that level, and Definition~\ref{def:quasi_dynamics} assigns that ball one of the three types. Reading them in order of increasing $n$ gives the \emph{$A/E/I$ word} of $a$, a string of length $N-1$ over the alphabet $\{A,E,I\}$, where $A$ records $n$-contracting, $E$ records $n$-expanding and $I$ records $n$-isometric.
\end{definition}

\noindent The word is the scale-resolved signature of the configuration under $\pi$, and it is finer than the fixed-point equation and than any label read at a single resolution: two configurations with $f(a)=a$ are alike as solutions of that equation, while their words may differ from the first level at which their balls separate. Section~\ref{sec:athaliana} exhibits a map whose fixed points split into two classes that agree at the first two resolutions and part in their subsequent $A/E/I$ profiles.

\begin{figure}[htbp]
\centering
\includegraphics[width=\textwidth]{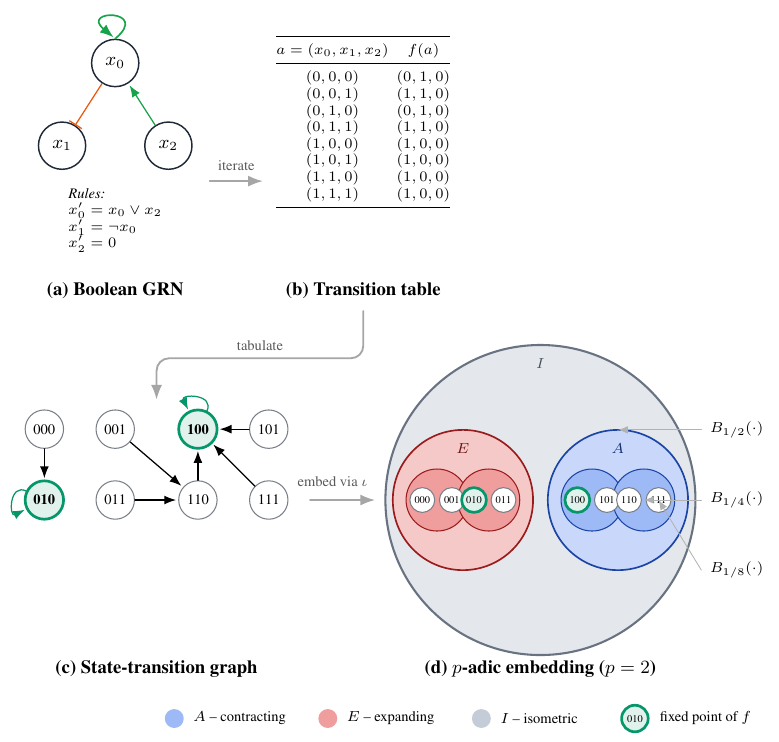}
\caption{From a Boolean network to a $p$-adic ball partition with ball-level labels ($p=2$, $N=3$). The first panels show a small Boolean network, its transition table, and its state-transition graph. The last panel shows the induced partition of $\mathbb{Z}_2$ into eight leaf-balls. The colors record the $A/E/I$ classification of Definition~\ref{def:quasi_dynamics}.}
\label{fig:grn_to_padic_pipeline}
\end{figure}

Whenever the image ball lies in the same ball chain as $\BCP_{1/p^n}(m)$, the classification of Definition~\ref{def:quasi_dynamics} takes the set-inclusion form $F_n(\BCP_{1/p^n}(m)) \subsetneq \BCP_{1/p^n}(m)$ for contracting, $\supsetneq$ for expanding, and $=$ for isometric.

\begin{remark}[Invariance of the classification]
\label{cor:invariance_Cn}
Two $n$-th approximations $F_n$ and $F'_n=F_n+h$ have the same ball images, $F'_n(\BCP_{1/p^n}(m))=F_n(\BCP_{1/p^n}(m))$ for every $m$ (Proposition~\ref{prop:perturbation}), so the $n$-contracting/expanding/isometric classification, and with it the scores of Section~\ref{sec:stability_measure}, survives any change of approximant.
\end{remark}

\begin{remark}[Ball-Level Classification and $\mathbb{Z}_p$]
\label{rem:quasi_Zp}
The classification of Definition~\ref{def:quasi_dynamics} is read on the observable balls $\BCP_{1/p^n}(m) \cap \mathbb{Z}_p$, through the observable image of $F_n$. A fixed point forced by a contraction on a $\mathbb{C}_p$-ball may fail to be observable, while the ball-level label, formulated on the observable image, stays unambiguous. For fixed points both the ``$n$-contracting'' and the ``$n$-expanding'' labels need one further hypothesis, since in either case the image ball may be disjoint from the source, and on the network of Section~\ref{sec:athaliana} it usually is. Once the image meets the source, ultrametricity nests the two. In the contracting case the image lies inside the source and Lemma~\ref{lem:closed_ball_bridge}(a) supplies an attracting fixed point; in the expanding case the source lies inside its own image and Lemma~\ref{lem:closed_ball_bridge}(b) supplies a repelling one with multiplier of absolute value $\Lambda_{n,m}>1$.

The label is indexed by the observable source ball $B_{n,m}=\BCP_{1/p^n}(m)\cap\mathbb{Z}_p$, while the radius comparison that defines it is made on the analytic image $F_n(\BCP_{1/p^n}(m))=\BCP_{t_{n,m}}(\beta_{n,m})\subset\mathbb{C}_p$. Since $\beta_{n,m}\in\mathbb{Z}_p$, that target ball has a nonempty observable trace of the same radius. No equality between the pointwise image $F_n(B_{n,m})$ and that trace is asserted or needed.
\end{remark}

\begin{remark}[$n$-trapping balls, $n$-covering balls, and the isometric case]
\label{rem:fixed_points}
Call an $n$-contracting ball that contains its own image, $F_n(\BCP_{1/p^n}(m)) \subsetneq \BCP_{1/p^n}(m)$, an \emph{$n$-trapping ball}, and an $n$-expanding ball contained in its own image an \emph{$n$-covering ball}. For the first, the gluing construction keeps the poles of $F_n$ outside the ball. Write the source ball about $\beta_{n,m}$, which is legitimate because that center lies in the source ball and every point of a non-Archimedean ball is a center. The hypothesis of Lemma~\ref{lem:closed_ball_bridge}(a) then holds at $\beta_{n,m}$ with $r'=t_{n,m}<p^{-n}$, and that lemma gives a unique attracting fixed point. For the second, injectivity on the closed source ball (established in the last step of the proof of Theorem~\ref{thm:existence_approximations}, recorded in Remark~\ref{rem:existence_uniformity}, and preserved under the perturbations of Proposition~\ref{prop:perturbation}, all by~Lemma~\ref{lem:image_ball_perturbation}) puts us in Lemma~\ref{lem:closed_ball_bridge}(b) with $t=t_{n,m}>r=p^{-n}$, which gives a unique repelling fixed point with multiplier of absolute value $t/r=\Lambda_{n,m}>1$. These are the resolution-$n$ analogues of attracting and repelling behavior, and in both cases Lemma~\ref{lem:closed_ball_bridge} gives the multiplier exactly, $|F_n'(b)|_p=\Lambda_{n,m}$, the constructed approximants being bijective from the source ball onto their image ball, which is what part~(b) of that lemma needs. Moreover, in the trapping case every point of the source ball converges to that fixed point under iteration of $F_n$, by Lemma~\ref{lem:closed_ball_bridge}(a). The trapping ball is therefore a trapping neighborhood contained in the analytic basin of its fixed point, and is the whole basin only for the dynamics restricted to it. A small image radius alone does not suffice: an $n$-contracting ball whose image lies in a \emph{different} ball carries no fixed point and feeds contractively into another basin. Neither the fixed points nor their type enter the scores, which are combinatorial. The fixed points are those of the analytic map $F_n$ in $\mathbb{C}_p$ and need not be observable (Remark~\ref{rem:quasi_Zp}), and when the ball chain contains a fixed configuration of $f$, that configuration labels an observable microball inside the canonical target ball $\BCP_{t_{n,m}}(\beta_{n,m})$, which is where the attracting fixed point lies. The point need not sit in that particular microball unless $M_{n,m}=N$, in which case the target ball is a single microball and the two agree.

The isometric case differs in one respect. If an $n$-isometric ball is carried onto itself, Lemma~\ref{lem:closed_ball_bridge}(b) with $t=r$ makes $F_n$ an isometry of it, so any fixed point is indifferent, with $|F_n'(x_0)|_p=1=\Lambda_{n,m}$. Existence is not forced: the attracting and repelling conclusions above both rest on a strict nesting of the source ball and its image, and here the two coincide.
\end{remark}

\begin{remark}[Scope and dependence on the coordinate ordering]
\label{rem:scope_analysis}
\label{Rem:order}
The scores and the optimal-ordering statements are for a fixed resolution $N$ and a fixed map $f$. Under that convention the ball partition is independent of the choice of approximants (Remark~\ref{cor:invariance_Cn}), so everything downstream depends on the coordinate ordering alone.
\end{remark}

%% file: 05_toy_example.tex
\section{Illustrative Example: A Toy Model}
\label{sec:toy_example}

We illustrate the framework on the smallest instance rich enough to show the construction in full, $p=2$, $N=4$ (coordinates $x_0,x_1,x_2,x_3$, so the configuration space $\mathcal{C}$ has $16$ elements, with $(a_0,a_1,a_2,a_3)$ encoded as $m = a_0 + 2a_1 + 4a_2 + 8a_3$). Interpolated over $\mathbb{F}_2$, the component updates are the multilinear polynomials
\begin{equation}
\label{eq:toy_polys}
x_0' = x_0, \qquad x_1' = 0, \qquad x_2' = x_1 + x_0 x_1, \qquad x_3' = x_0 x_1 + x_0 x_2,
\end{equation}
equivalently $x_2' = \neg x_0 \wedge x_1$ and $x_3' = x_0 \wedge (x_1 \oplus x_2)$ in Boolean form. Read variable by variable. The coordinate $x_0$ is a self-latch ($x_0'=x_0$), so its value is frozen from the first step. The coordinate $x_1$ is switched off ($x_1'=0$), and $x_2,x_3$ are driven by $x_0,x_1$. The eventual fate of every trajectory is therefore governed by the initial value of $x_0$, while the transient still reads $x_1$.

Table~\ref{tab:toy_transitions_main} lists the map in full. There are two fixed points, $(0,0,0,0)$ and $(1,0,0,0)$ (integers $0$ and $1$), one for each value of the latch $x_0$. Their basins partition the $16$ states into two sets of eight, and every trajectory reaches its fixed point in at most two synchronous steps (through the intermediate states $m=4$ or $m=9$).

\begin{table}[htbp]
\centering
\caption{Transition table of the toy model, $f\colon m\mapsto m'$. Here $m'=f(m)$ is the image and $m^{\ast}=\lim_k f^{k}(m)$ is the fixed point the state reaches under iteration, attained after at most two steps. The latch $x_0$ (least significant bit) is preserved by $f$, so the two columns are exactly the two basins.}
\label{tab:toy_transitions_main}
\small
\begin{tabular}{@{}cccc@{\qquad}cccc@{}}
\toprule
\multicolumn{4}{c}{Basin of $0$ ($x_0=0$)} & \multicolumn{4}{c}{Basin of $1$ ($x_0=1$)}\\
\cmidrule(r{1.5em}){1-4}\cmidrule{5-8}
$m$ & $(x_0x_1x_2x_3)$ & $m'$ & $m^{\ast}$ & $m$ & $(x_0x_1x_2x_3)$ & $m'$ & $m^{\ast}$\\
\midrule
$0$  & $0000$ & $0$ & $0$ & $1$  & $1000$ & $1$ & $1$\\
$2$  & $0100$ & $4$ & $0$ & $3$  & $1100$ & $9$ & $1$\\
$4$  & $0010$ & $0$ & $0$ & $5$  & $1010$ & $9$ & $1$\\
$6$  & $0110$ & $4$ & $0$ & $7$  & $1110$ & $1$ & $1$\\
$8$  & $0001$ & $0$ & $0$ & $9$  & $1001$ & $1$ & $1$\\
$10$ & $0101$ & $4$ & $0$ & $11$ & $1101$ & $9$ & $1$\\
$12$ & $0011$ & $0$ & $0$ & $13$ & $1011$ & $9$ & $1$\\
$14$ & $0111$ & $4$ & $0$ & $15$ & $1111$ & $1$ & $1$\\
\bottomrule
\end{tabular}
\end{table}

One exponent computed in full shows the mechanism. At level $n=2$ the ball $m=1$ collects the configurations whose first two ordered digits read $(1,0)$, that is the integer centers $1,5,9,13$. Their images have centers $1,9,1,9$. These agree modulo $2^{3}$ and not modulo $2^{4}$, so
\[
M_{2,1}=3,\qquad t_{2,1}=2^{-3},\qquad \Lambda_{2,1}=\frac{t_{2,1}}{2^{-2}}=\tfrac12<1,
\]
and the ball is contracting. Every other entry of Table~\ref{tab:toy_hierarchical_data} is obtained the same way, by reading a column of images off Table~\ref{tab:toy_transitions_main} and finding the longest common prefix.

Following Section~\ref{sec:hierarchical_framework}, the hierarchical approximations $F_0,\ldots,F_4$ exist. At every level $n=1,2,3$ the image radius of each source ball is smaller than the source radius, so every interior ball is $n$-contracting, while the level-$0$ ball is $0$-isometric. This toy is the clean, purely canalizing case, chosen so that the mechanics are transparent before the mixed behavior of the \textit{A.\ thaliana} network in Section~\ref{sec:athaliana}. Table~\ref{tab:toy_hierarchical_data} records the image balls level by level, and the classification is read off by comparing image radius with source radius. The updates~\eqref{eq:toy_polys} are triangular for the displayed ordering (Definition~\ref{def:triangular_compatible}), and Section~\ref{sec:structure_functional} returns to this map once the scores are available.

\begin{table}[htbp]
\centering
\caption{Hierarchical image data for the toy model. At level $n$ the source ball $\BCP_{1/2^{\,n}}(m)$ has radius $1/2^{\,n}$. The listed image ball has a strictly smaller radius at every entry, which is what makes each ball $n$-contracting ($A$).}
\label{tab:toy_hierarchical_data}
\small
\begin{tabular}{@{}clll@{}}
\toprule
Level $n$ & Source radius & Source balls & Image balls (all of smaller radius) \\
\midrule
$1$ & $1/2$  & $m=0,1$ &
$\BCP_{1/4}(0),\ \BCP_{1/8}(1)$ \\
$2$ & $1/4$  & $m=0,1,2,3$ &
$\BCP_{1/16}(0),\ \BCP_{1/8}(1),\ \BCP_{1/16}(4),\ \BCP_{1/8}(1)$ \\
$3$ & $1/8$  & $m=0,\ldots,7$ &
all radii $1/16$, centers $0,1,4,9,0,9,4,1$ \\
\bottomrule
\end{tabular}
\end{table}

Only the chains through the finest balls $\BCP_{1/16}(0)$ and $\BCP_{1/16}(1)$ contain their own image, and each therefore carries a unique attracting fixed point of the corresponding approximant. That fixed point is the analytic one supplied by Remark~\ref{rem:fixed_points}. At the finest level, where $M_{n,m}=N$, it lies in the analytic microball $\BCP_{1/16}(m_a)$ whose observable trace $\BZP_{1/16}(m_a)$ is labeled by the discrete fixed configuration. The fixed point itself need not equal that configuration and need not belong to $\mathbb{Z}_p$ (Remark~\ref{rem:quasi_Zp}). The dependency graph, the state-transition graph, and the image-ball data determining the approximants $F_0,\ldots,F_3$, with $F_0$ and $F_1$ in closed form, are given in the Supplementary Material.

%% file: 06_stability_measure.tex
\section{Hierarchical Scores and the Variational Problem}
\label{sec:stability_measure}

The coordinate ordering fixes the embedding $\iota_\pi\colon\mathcal{C}\to\mathbb{Z}_p/p^N\mathbb{Z}_p$ and therefore the hierarchy of balls on which the dynamics is observed (Section~\ref{subsec:embedding}). This section defines the scale-resolved scores $\mu_E,\mu_A,\mu_I$ and states the variational problem over orderings; Section~\ref{sec:structure_functional} then proves what the expansion score measures. The analytic role of Section~\ref{sec:hierarchical_framework} is essential: Theorem~\ref{thm:existence_approximations} realizes the finite transition map by rational approximants over $\mathbb{C}_p$ whose ball images have the prescribed radii. The computation of the radii, however, is purely finite.

\paragraph{Canonical combinatorics.}
\label{subsec:encoding_canonical}

Recall from Section~\ref{subsec:rational_approximations} that the ordering gives each configuration $a\in\mathcal{C}$ an integer center $m_a=\sum_{i=0}^{N-1}a_i p^i$, that $\operatorname{trunc}_n(a):=m_a\bmod p^n$, and that $M_{n,m}$ and $t_{n,m}=p^{-M_{n,m}}$ are defined by \eqref{eq:Mnm_definition}--\eqref{eq:tnm_definition}.

The coordinates of $\mathcal{C}=\mathbb{F}_p^N$ are labeled by $\{0,1,\dots,N-1\}$, as in~\eqref{eq:padic_encoding}, so that a configuration is a tuple $a=(a_0,\dots,a_{N-1})$. An \emph{ordering} is a permutation $\pi\in S_N$ of that index set, read as placing the coordinate $\pi(i)$ at the digit position $i$. It determines the integer center
\begin{equation}
\label{eq:mapi_definition}
m_a=\sum_{i=0}^{N-1}a_{\pi(i)}\,p^{\,i},
\end{equation}
which is \eqref{eq:padic_encoding} written with the ordering displayed, and through it the embedding $\iota_\pi$, the truncations $\operatorname{trunc}_n$, the exponents $M_{n,m}$ and every quantity built from them. We write $\mu_E(\pi)$, $\mu_A(\pi)$, $\mu_I(\pi)$ when the dependence on the ordering is the point, and drop the argument when a single ordering is fixed. Since $0\le a_{\pi(i)}\le p-1$, the base-$p$ digits of $m_a$ are exactly the values $a_{\pi(0)},\dots,a_{\pi(N-1)}$, so $\operatorname{trunc}_n(a)$ records the values of $a$ on the first $n$ coordinates of the ordering and nothing else. That last observation is what Section~\ref{sec:structure_functional} turns into a structure theorem.

We keep the notation of Definition~\ref{def:quasi_dynamics}: for $0\le n\le N-1$ and $m\in\{0,\ldots,p^n-1\}$,
\[
\Lambda_{n,m}=\frac{t_{n,m}}{p^{-n}}=p^{n-M_{n,m}}.
\]
Thus $M_{n,m}$ is the common prefix length of the images of all configurations lying in the level-$n$ ball indexed by $m$, $t_{n,m}$ is the prescribed image radius of an $n$-th approximant, and $\Lambda_{n,m}$ is the coarse multiplier of Definition~\ref{def:quasi_dynamics}.

At level $n$ we classify the $p^n$ balls by
\begin{equation}
\label{eq:EAI_sets}
E(n):=\{m:M_{n,m}<n\},\qquad
I(n):=\{m:M_{n,m}=n\},\qquad
A(n):=\{m:M_{n,m}>n\}.
\end{equation}
Equivalently, $E$ means $\Lambda_{n,m}>1$, $I$ means $\Lambda_{n,m}=1$, and $A$ means $\Lambda_{n,m}<1$. These letters are labels for the geometric ball type: expanding, isometric, and contracting. The attracting/repelling terminology is reserved for the same-ball-chain situation of Remark~\ref{rem:fixed_points}.

\paragraph{Scores.}
\label{subsec:definition_stability}

The weighted scores of an ordering are
\begin{equation}
\label{eq:mu_definition}
\mu_E(\pi):=\sum_{n=0}^{N-1}|E(n)|p^{N-n},\qquad
\mu_A(\pi):=\sum_{n=0}^{N-1}|A(n)|p^{N-n},\qquad
\mu_I(\pi):=\sum_{n=0}^{N-1}|I(n)|p^{N-n}.
\end{equation}
The primary score used for optimization is $\mu:=\mu_E$, the \emph{expansion functional}. The singleton level $n=N$ is left out, because each level-$N$ ball corresponds to a single configuration and the classification is vacuous there. Although the three scores share the range $0\le n\le N-1$, the root never contributes to $\mu_E$, since $E(0)=\emptyset$, so the total score budget is $N\,p^N$ while the expansion budget is $(N-1)\,p^N$. Remark~\ref{rem:range_endpoints} discusses both endpoints in detail.

The weight $p^{N-n}$ is canonical, and it is one number read in two ways. A level-$n$ cylinder contains exactly $p^{N-n}$ level-$N$ cylinders, and equivalently
\begin{equation}
\label{eq:weight_is_haar}
p^{\,N-n}=p^N\lambda_{\mathrm{H}}(B_{n,m})
\end{equation}
for the normalized Haar measure $\lambda_{\mathrm{H}}$ on $\mathbb{Z}_p$ recalled in Section~\ref{sec:setup}. Each expanding ball is therefore weighted by the share of the observable state space it stands for, and the factor $p^N$ is there only to keep the score an integer. A coarse ball weighs more than a fine one because it represents more configurations, but no scale is privileged in the aggregate: there are $p^n$ balls at level $n$, so every complete level carries the same maximal budget $p^np^{N-n}=p^N$. Section~\ref{sec:structure_functional} turns~\eqref{eq:weight_is_haar} into an integral representation of $\mu_E$ and reads off what the integrand counts.

\begin{lemma}[Level-wise partition]
\label{lem:level_partition}
For each $0\le n\le N-1$,
\[
|E(n)|+|A(n)|+|I(n)|=p^n.
\]
\end{lemma}
\begin{proof}
The level-$n$ balls partition $\mathbb{Z}_p$, and exactly one of $M_{n,m}<n$, $M_{n,m}=n$, or $M_{n,m}>n$ holds for each index $m$, which is the assertion.
\end{proof}

\begin{lemma}[Control identity]
\label{lem:mu_checksum}
The scores satisfy
\[
\mu_E+\mu_A+\mu_I=N\,p^N.
\]
\end{lemma}
\begin{proof}
Sum the identity of Lemma~\ref{lem:level_partition} after multiplying the level-$n$ equality by $p^{N-n}$, over the $N$ levels $n=0,\ldots,N-1$.
\end{proof}

\begin{remark}[The two endpoints of the scale range]
\label{rem:range_endpoints}
The scores sum over the $N$ resolution levels $n=0,\ldots,N-1$, and the two endpoints behave differently. At the root, $n=0$, the single ball is all of $\mathbb{Z}_p$. It never expands, since $\operatorname{trunc}_0\circ f$ is constant, so $E(0)=\emptyset$ and both the expansion score and its minimizers are the same as under the narrower range $1\le n\le N-1$. Its class is isometric or contracting according to whether $M_{0,0}=0$ or $M_{0,0}>0$, and $M_{0,0}>0$ holds exactly when every image shares its leading ordered digit, that is, when the leading coordinate $\pi(0)$ has a constant transition. The root therefore contributes $p^N$ to $\mu_A$ when $\pi(0)$ is a constant coordinate and to $\mu_I$ otherwise, and we retain it. At the singleton level $n=N$, by contrast, each level-$N$ ball corresponds to a single configuration and $M_{N,m}=N$, so all $p^N$ balls are isometric for every $f$ and every $\pi$. That level would add the constant $p^N$ to $\mu_I$ alone, independently of the map and the ordering, so it distinguishes nothing and shifts no minimizer, and we exclude it. The widened range is what endows $\mu_A$ and $\mu_I$ with their root term. Theorem~\ref{thm:muE_zero_lipschitz} explains the asymmetry: the lift can fail to be $1$-Lipschitz exactly at the scales $1\le n\le N-1$, which is the range the expansion score sums over.
\end{remark}

\begin{remark}[Multi-objective trade-off]
\label{rem:tradeoff}
The fixed budget is a useful consistency check and has one substantive consequence: minimizing $\mu_E$ is equivalent to maximizing $\mu_A+\mu_I$, not to maximizing $\mu_A$ alone. Thus the $\mu_E$- and $\mu_A$-optimal orderings need not coincide. That comparison is treated separately in~\cite{perez2026padic_applications}.
\end{remark}

\paragraph{The same map at every value of the functional.}
Sixteen states are enough to see the whole range of $\mu_E$. Running the score over all $4!=24$ orderings of this one map gives values
\[
\mu_E\in\{0,\,8,\,16,\,20,\,24,\,28,\,32,\,36,\,40,\,44,\,48\},
\]
from the minimum $0$, attained at the two topological orders of its dependency digraph (Corollary~\ref{cor:min_zero_dag}), to $48=(N-1)2^{N}$, which is the saturation value of Theorem~\ref{thm:muE_saturation}(a). The maximum is attained at exactly two orderings, $(x_2,x_3,x_0,x_1)$ and $(x_3,x_2,x_0,x_1)$, where every ball expands at every interior level: $|E(1)|=2$, $|E(2)|=4$, $|E(3)|=8$, each level contributing its full budget $2^{N}=16$. Both place the two driven coordinates $x_2,x_3$ at the coarsest positions, so the first two nontrivial levels are formed without the latch $x_0$ that determines the eventual fixed point. The same finite map therefore realizes both extremes of the functional, and the variational problem over orderings is already nontrivial at $N=4$. The richer mix of contracting, expanding, and isometric behavior across scales appears in the \textit{A.\ thaliana} application of Section~\ref{sec:athaliana}.

\begin{remark}[Saturated dynamics]
\label{rem:saturated}
When $\mu_E=(N-1)p^N$ for every ordering, every ball expands at every interior level and the variational principle detects no exploitable hierarchy. Theorem~\ref{thm:muE_saturation}(b) characterizes the maps for which this happens, and the Supplementary Material identifies them among the elementary cellular automata.
\end{remark}

\paragraph{Intrinsic computability.}

Recall that $B_{n,m}=\BZP_{1/p^n}(m)\subset\mathbb{Z}_p$ is the observable ball of Section~\ref{sec:hierarchical_framework}.

\begin{proposition}[Intrinsic combinatorial data]
\label{thm:intrinsic_mu}
The quantities $M_{n,m}$, $t_{n,m}$ and $\Lambda_{n,m}$, the sets $E(n)$, $A(n)$ and $I(n)$, and the scores $\mu_E$, $\mu_A$ and $\mu_I$ are functions of the pair $(f,\iota)$ alone.
\end{proposition}

\begin{proof}
By~\eqref{eq:Mnm_definition}, $M_{n,m}$ is the common prefix length of the image centers
\[
\{\operatorname{trunc}_N(f(a))\;:\;\operatorname{trunc}_n(a)=m\},
\]
a set read off from $f$ and from the ordering. It follows from~\eqref{eq:tnm_definition}--\eqref{eq:mu_definition} that the remaining quantities are explicit functions of the $M_{n,m}$.
\end{proof}

\noindent Nothing analytic enters that statement. What the analytic layer adds is the next proposition, where the two meet: these numbers are the radii that rational maps realize on the balls of the hierarchy. It is the one place in the paper where the finite computation and the non-Archimedean geometry have to agree, and they do.

\begin{proposition}[The canonical radius is attained]
\label{prop:radius_attained}
For every interpreter $\phi$ and every $n$-th approximant $F_n$,
\begin{equation}
\label{eq:tnm_image_radius_bound}
\diam\!\bigl(\phi(B_{n,m})\bigr)\le t_{n,m},\qquad
\diam\!\bigl(F_n(B_{n,m})\bigr)\le t_{n,m}.
\end{equation}
The first bound is an equality for every interpreter whenever $M_{n,m}<N$, and for the exact interpreters of Proposition~\ref{prop:exact_existence} at every $(n,m)$. The second is an equality for the approximants constructed in Theorem~\ref{thm:existence_approximations}.
\end{proposition}

\begin{proof}
For the bounds, if $x\in\BCP(a)\cap\mathbb{Z}_p$ and $y\in\BCP(b)\cap\mathbb{Z}_p$ with $\operatorname{trunc}_n(a)=\operatorname{trunc}_n(b)=m$, then an interpreter sends $x$ into $\BCP(f(a))$ and $y$ into $\BCP(f(b))$. When $f(a)=f(b)$, the two images lie in one ball of radius $p^{-N}$, so their distance is at most $p^{-N}$. When $f(a)\ne f(b)$, disjointness of the target micro-balls gives
\[
|\phi(x)-\phi(y)|_p=|m_{f(a)}-m_{f(b)}|_p=p^{-j},
\]
where $j$ is the first digit at which the two image centers differ. Taking the supremum over all such pairs gives at most $p^{-M_{n,m}}=t_{n,m}$. For $F_n$ the bound is immediate from Definition~\ref{def:nth_approximation}, since $F_n(B_{n,m})\subseteq F_n(\BCP_{1/p^n}(m))=\BCP_{t_{n,m}}(\beta_{n,m})$.

We turn to the equality. Suppose first $M_{n,m}<N$. By maximality in~\eqref{eq:Mnm_definition} there are $a,b$ in the ball with $m_{f(a)}\equiv m_{f(b)}\pmod{p^{M_{n,m}}}$ but not modulo $p^{M_{n,m}+1}$, so $f(a)\neq f(b)$ and the first differing digit is $j=M_{n,m}$. The centers $m_a,m_b$ lie in $B_{n,m}$, and the displayed identity gives $|\phi(m_a)-\phi(m_b)|_p=t_{n,m}$. This step uses only Definition~\ref{def:interprets}, so at these levels the equality holds for every interpreter.

Suppose now $M_{n,m}=N$, that is, $f$ is constant on $B_{n,m}$. All target micro-balls coincide and no pair of source configurations separates the images, so the equality is not a consequence of exactness alone and is read off the construction. Let $\phi$ be as in Proposition~\ref{prop:exact_existence}, with local isometries $g_a(z)=A_a(z-m_a)+m_{f(a)}$, $|A_a|_p=1$, and $\sup_{\BCP(a)}|\phi-g_a|_p<\varepsilon<1/p^N$. Fix $a$ with $\operatorname{trunc}_n(a)=m$ and take $x=m_a$, $y=m_a+p^N$, both in $\BCP(a)\cap\mathbb{Z}_p\subseteq B_{n,m}$. Then $|g_a(x)-g_a(y)|_p=|x-y|_p=p^{-N}$ while $|(\phi-g_a)(x)-(\phi-g_a)(y)|_p<p^{-N}$, and the strong triangle inequality leaves the isometric term in charge, giving $|\phi(x)-\phi(y)|_p=p^{-N}=t_{n,m}$.

For the approximants the same mechanism is needed at every $M_{n,m}$, since $F_n$ controls level-$n$ balls only. In the notation of the proof of Theorem~\ref{thm:existence_approximations}, on $\BCP_{1/p^n}(m)$ we may write $F_n(z)=A_m(z-m)+\beta_{n,m}+h_m(z)$ with $|A_m|_p=\Lambda_{n,m}=t_{n,m}p^{n}$, where $h_m$ collects the deviation of $\hat g_m$ from its affine part in Case~B, of sup norm at most $t_{n,m}/p$, together with $F_n-\hat g_m$, of sup norm below $\varepsilon<\min_{m'}t_{n,m'}$ by~\eqref{eq:Fn_sup_bound}. In Case~A the deviation is absent, and for $n=0$ one takes $h_0=0$. In every case $\sup_{\BCP_{1/p^n}(m)}|h_m|_p<t_{n,m}$. Taking $x=m$ and $y=m+p^{n}$, both in $B_{n,m}$, gives $|A_m(x-y)|_p=\Lambda_{n,m}p^{-n}=t_{n,m}$ and $|h_m(x)-h_m(y)|_p<t_{n,m}$, hence $|F_n(x)-F_n(y)|_p=t_{n,m}$, as claimed.
\end{proof}

\begin{remark}[Exact versus generic interpreters]
\label{rem:exact_vs_generic}
For the exact interpreters constructed in Proposition~\ref{prop:exact_existence}, the diameter of the image of the observable ball $B_{n,m}$ equals $t_{n,m}$ (Proposition~\ref{prop:radius_attained}). For a generic interpreter satisfying only the inclusion of Definition~\ref{def:interprets}, this observable diameter may be smaller, for instance when $f$ is constant on a ball. The scores are unchanged, by Proposition~\ref{thm:intrinsic_mu}.
\end{remark}

\paragraph{Analytic interpretation.}

The coarse multiplier $\Lambda_{n,m}$ is the ratio of image radius to source radius for the approximant supplied by Theorem~\ref{thm:existence_approximations}. This is the ball-level analogue of the multiplier of a fixed point in $p$-adic dynamics. An $n$-trapping ball carries an attracting fixed point of the approximant and an $n$-covering ball a repelling one (Remark~\ref{rem:fixed_points}), by Lemma~\ref{lem:closed_ball_bridge}. Without the rational dynamics over $\mathbb{C}_p$, the number $M_{n,m}$ would be only a common-prefix statistic; Theorem~\ref{thm:existence_approximations} is what turns the statistic into a statement about images of non-Archimedean balls.

\paragraph{The variational problem.}
\label{subsec:variational_problem}

The object of the paper is the minimization
\begin{equation}
\label{eq:variational_problem}
\pi^{*}\in\arg\min_{\pi\in S_N}\mu_E(\pi),
\end{equation}
taken over the finite group $S_N$. The word ``variational'' is used in the finite sense: a functional on a space of admissible configurations, here the $N!$ orderings, and its minimizers. What the minimization is asking for becomes visible once the two extremes of $\mu_E$ and its integral form are known, which is the business of Section~\ref{sec:structure_functional}.

\paragraph{Finding optimal orderings.}
\label{subsec:optimal_orderings}
\label{app:computational_basic}

For a fixed ordering, evaluating $\mu_E$ requires computing the common-prefix length \eqref{eq:Mnm_definition} at each interior level, so the cost is $O(Np^N)$. For small $N$, minimization over all $N!$ orderings can be exhaustive. For the $N=13$ \textit{A. thaliana} network, the minimizers reported in Section~\ref{sec:athaliana} were certified by branch-and-bound over all $13!$ orderings, with logs in the reproducibility bundle~\cite{padic_grn_bundle}.

%% file: 06b_structure_functional.tex
\section{Structure of the Expansion Functional}
\label{sec:structure_functional}
\label{subsec:vanishing_locus}

The definition~\eqref{eq:mu_definition} is a weighted count of expanding balls, and so far
nothing says why that particular count, with those particular weights, is the right object to
minimize. This section answers the question. It computes the two extremes of the functional,
identifies what the general term of the sum counts, and turns the canonical weights of
\eqref{eq:weight_is_haar} into an integral representation whose integrand has a pointwise
meaning. The vanishing locus of $\mu_E$ turns out to be a classical object, and once that is
known the rest follows.

We first name the combinatorial condition that governs everything below.

\begin{definition}[Prefixes and autonomous blocks]
\label{def:autonomous_block}
Let $\pi$ be an ordering. For $1\le n\le N-1$ the \emph{level-$n$ prefix} of $\pi$ is
the set of coordinates occupying the $n$ coarsest digit positions,
\[
J^\pi_n:=\{\pi(0),\pi(1),\dots,\pi(n-1)\}\subseteq\{0,\dots,N-1\}.
\]
Let $J\subseteq\{0,\dots,N-1\}$ be any set of coordinates and let $u\in\mathbb{F}_p^{J}$ be
an assignment of values to them. Write
\[
\mathcal{C}(J,u):=\{a\in\mathcal{C}\ :\ a|_{J}=u\}
\]
for the set of configurations extending $u$. We say that $J$ is \emph{autonomous at $u$}
if the map $a\mapsto f(a)|_{J}$ is constant on $\mathcal{C}(J,u)$, and that $J$ is
\emph{autonomous} if it is autonomous at every $u\in\mathbb{F}_p^{J}$, that is, if
$f(a)|_{J}$ is a function of $a|_{J}$ alone.
\end{definition}

Autonomy at $u$ is a local statement: knowing that the coordinates in $J$ currently read
$u$ already determines what they will read next, whatever the other coordinates are
doing. Autonomy without qualification is the same statement at every partial state, and
says that the subsystem carried by $J$ evolves without consulting its complement. Three
minimal examples fix the idea, all read off the toy model of
Section~\ref{sec:toy_example}, whose updates are~\eqref{eq:toy_polys}. The block
$J=\{x_0\}$ is autonomous, because $x_0'=x_0$. The block $J=\{x_2\}$ is autonomous at no
$u$ at all, because $x_2'=x_1+x_0x_1$ takes both values on every fiber $x_2=c$, the
coordinates $x_0$ and $x_1$ being free there. The block $J=\{x_1,x_2\}$ sits in between:
on the fiber over $u$ the pair $(x_1',x_2')$ equals $(0,\,u_{x_1}(1+x_0))$, which is
constant when $u_{x_1}=0$ and takes both values when $u_{x_1}=1$, so $J$ is autonomous at
exactly two of its four assignments. That intermediate case is what the local version
records, and it is what $\mu_E$ counts.

\begin{lemma}[The admissible exponents form an initial segment]
\label{lem:initial_segment}
Fix an ordering $\pi$, a level $0\le n\le N$ and an index $m$, and let
\[
\begin{aligned}
K_{n,m}:=\Bigl\{k\in\{0,\dots,N\}\ :\ &m_{f(a)}\equiv m_{f(b)}\!\!\pmod{p^{k}}\\
&\text{for all } a,b\in\mathcal{C} \text{ with }
\operatorname{trunc}_n(a)=\operatorname{trunc}_n(b)=m\Bigr\}
\end{aligned}
\]
be the set of exponents admissible in~\eqref{eq:Mnm_definition}, so that
$M_{n,m}=\max K_{n,m}$. Then $K_{n,m}$ is an \emph{initial segment} of
$\{0,\dots,N\}$, meaning that it is nonempty and closed downwards: if $k\in K_{n,m}$ and
$0\le k'\le k$ then $k'\in K_{n,m}$. Consequently $K_{n,m}=\{0,1,\dots,M_{n,m}\}$ and,
for every $0\le n\le N$,
\begin{equation}
\label{eq:Mnm_criterion}
M_{n,m}\ge n
\iff
\text{the centers } m_{f(a)},\ \operatorname{trunc}_n(a)=m,\ \text{all agree modulo } p^{n}.
\end{equation}
\end{lemma}

\begin{proof}
The exponent $k=0$ lies in $K_{n,m}$, since every pair of integers is congruent modulo
$p^{0}=1$, so $K_{n,m}\neq\emptyset$. If $k\in K_{n,m}$ and $0\le k'\le k$ then
$p^{k'}\mid p^{k}$, so a congruence modulo $p^{k}$ implies the congruence modulo
$p^{k'}$, and the defining condition for $k'$ holds verbatim. Also $k'\in\{0,\dots,N\}$.
Hence $K_{n,m}$ is closed downwards. The cap $k\le N$ never interferes: the centers $m_{f(a)}$ lie in $\{0,\dots,p^{N}-1\}$, so $N\in K_{n,m}$ exactly when $f$ is constant on $B_{n,m}$, and then $M_{n,m}=N\ge n$ still gives the right answer at the levels $n\le N-1$ where the criterion is used. A nonempty downward closed subset of the finite
chain $\{0,\dots,N\}$ is the set of all elements up to its maximum, which is
$M_{n,m}$ by definition, so $K_{n,m}=\{0,\dots,M_{n,m}\}$.

For~\eqref{eq:Mnm_criterion}, note that the right-hand side is precisely the statement
$n\in K_{n,m}$, which is legitimate because $n\le N$. If $M_{n,m}\ge n$ then
$n\in\{0,\dots,M_{n,m}\}=K_{n,m}$. Conversely if $n\in K_{n,m}$ then
$M_{n,m}=\max K_{n,m}\ge n$, as required.
\end{proof}

\begin{theorem}[Combinatorial representation of the expansion functional]
\label{lem:balls_are_blocks}
Fix an ordering $\pi$ and a level $1\le n\le N-1$. The map
\begin{equation}
\label{eq:ball_block_bijection}
u\longmapsto m(u):=\sum_{i=0}^{n-1}u_{\pi(i)}\,p^{\,i}
\end{equation}
is a bijection from $\mathbb{F}_p^{J^\pi_n}$ onto $\{0,\dots,p^{n}-1\}$, it identifies
the fiber $\mathcal{C}(J^\pi_n,u)$ with the set of configurations lying in the level-$n$
ball $B_{n,m(u)}$, and
\begin{equation}
\label{eq:E_is_nonautonomy}
m(u)\in E(n)\iff J^\pi_n \text{ is not autonomous at } u .
\end{equation}
Consequently
\begin{equation}
\label{eq:mu_master_formula}
\mu_E(\pi)=\sum_{n=1}^{N-1}p^{\,N-n}\cdot
\#\bigl\{u\in\mathbb{F}_p^{J^\pi_n}\ :\ J^\pi_n \text{ is not autonomous at } u\bigr\}.
\end{equation}
\end{theorem}

\begin{proof}
By~\eqref{eq:mapi_definition} the base-$p$ digits of $m_a$ are the values
$a_{\pi(0)},\dots,a_{\pi(N-1)}$, so
$\operatorname{trunc}_n(a)=m_a\bmod p^{n}=\sum_{i<n}a_{\pi(i)}p^{\,i}$. The
map~\eqref{eq:ball_block_bijection} is therefore the restriction of $a\mapsto
\operatorname{trunc}_n(a)$ to the data $a|_{J^\pi_n}$, and it is a bijection because
base-$p$ expansion with $n$ digits is one. In particular
$\operatorname{trunc}_n(a)=m(u)$ if and only if $a|_{J^\pi_n}=u$, which is the asserted
identification of the fiber with the ball.

The same computation applied to the image configurations shows that
$\operatorname{trunc}_n(m_{f(a)})=\sum_{i<n}f(a)_{\pi(i)}p^{\,i}$ determines and is
determined by the restriction $f(a)|_{J^\pi_n}$. Hence the centers $m_{f(a)}$, for $a$
in the ball, all agree modulo $p^{n}$ if and only if $f(a)|_{J^\pi_n}$ is constant on
the fiber $\mathcal{C}(J^\pi_n,u)$, that is, if and only if $J^\pi_n$ is autonomous at
$u$. By Lemma~\ref{lem:initial_segment} the left-hand side is the statement
$M_{n,m(u)}\ge n$, whose negation is $m(u)\in E(n)$ by~\eqref{eq:EAI_sets}. This
proves~\eqref{eq:E_is_nonautonomy}.

Formula~\eqref{eq:mu_master_formula} is now~\eqref{eq:mu_definition} with $|E(n)|$
rewritten through the bijection, the level $n=0$ being omitted because $E(0)=\emptyset$
(Remark~\ref{rem:range_endpoints}).
\end{proof}

Formula~\eqref{eq:mu_master_formula} is the representation the rest of the paper works
from. It turns a qualitative condition into a quantity graded by resolution, and every
statement below is read off it.

We now name the classical conditions that the vanishing of $\mu_E$ turns out to express,
and the $p$-adic map on which they are stated. The equivalences among triangularity,
compatibility and the $1$-Lipschitz condition are standard in the $p$-adic automata
literature (\cite[Theorem~2.10, Corollaries~2.11--2.12, Proposition~2.13]{anashin2021automata_chapter};
\cite[\S3.8]{Anashin2009}). What
Theorem~\ref{thm:muE_zero_lipschitz} adds is that this classical locus is the zero set of
the graded obstruction~\eqref{eq:mu_master_formula}, and that the finitely many interior
scales $1\le n\le N-1$ already decide the matter, the root and the scales below $p^{-N}$
being automatic.

\begin{definition}[Triangular maps, compatible maps, and the lift]
\label{def:triangular_compatible}
Let $\pi$ be an ordering.
\begin{enumerate}[label=\textup{(\alph*)}]
  \item The map $f$ is \emph{triangular for $\pi$} if, for every $0\le i\le N-1$, the
        value $f(a)_{\pi(i)}$ is determined by $a_{\pi(0)},\dots,a_{\pi(i)}$, that is,
        if there are functions $\psi_i\colon\mathbb{F}_p^{\,i+1}\to\mathbb{F}_p$ with
        $f(a)_{\pi(i)}=\psi_i\bigl(a_{\pi(0)},\dots,a_{\pi(i)}\bigr)$ for all
        $a\in\mathcal{C}$. In words, a coordinate may look at the coordinates placed
        before it and at itself, and at nothing that comes later. The name is
        from~\cite[Definition~3.37]{Anashin2009}: listing the components as
        $\psi_0(a_{\pi(0)}),\ \psi_1(a_{\pi(0)},a_{\pi(1)}),\ \psi_2(a_{\pi(0)},a_{\pi(1)},a_{\pi(2)}),\dots$
        displays their dependencies in a triangular array, the $i$-th taking $i+1$
        arguments.
  \item Since $a\mapsto m_a$ is a bijection of $\mathcal{C}$ onto
        $\{0,\dots,p^{N}-1\}$, the map $f$ induces
        \[
        F_\pi\colon\mathbb{Z}/p^{N}\mathbb{Z}\to\mathbb{Z}/p^{N}\mathbb{Z},
        \qquad F_\pi(m_a):=m_{f(a)} .
        \]
        We call $F_\pi$ \emph{compatible} if $x\equiv y\pmod{p^{n}}$ implies
        $F_\pi(x)\equiv F_\pi(y)\pmod{p^{n}}$ for every $0\le n\le N$, that is, if
        $F_\pi$ descends to every quotient $\mathbb{Z}/p^{n}\mathbb{Z}$.
  \item The \emph{lift} of $f$ along $\pi$ is the locally constant map
        \[
        \begin{aligned}
        &\widehat{F}_\pi\colon\mathbb{Z}_p\to\mathbb{Z}_p,\qquad
        \widehat{F}_\pi(z):=m_{f(a)},\\
        &\text{where $a\in\mathcal{C}$ is the unique configuration with }
        m_a=\operatorname{trunc}_N(z).
        \end{aligned}
        \]
        It is the map that reads a $p$-adic integer at resolution $N$, applies $f$ to the
        configuration it sees, and returns the integer center of the image.
\end{enumerate}
\end{definition}

Compatibility and the $1$-Lipschitz condition are the same requirement written twice.
A map $g\colon\mathbb{Z}_p\to\mathbb{Z}_p$ is $1$-Lipschitz when
$|g(z)-g(w)|_p\le|z-w|_p$ for all $z,w$, and since $|z-w|_p\le p^{-n}$ says exactly that
$z\equiv w\pmod{p^{n}}$, the inequality at the scale $p^{-n}$ is exactly the implication
in Definition~\ref{def:triangular_compatible}(b) at that $n$. The only thing to check,
and it is the content of the equivalence (v)~$\Leftrightarrow$~(vi) below, is that for
the lift $\widehat{F}_\pi$ the scales $n\ge N$ are automatic.

\begin{theorem}[Automaton locus of $\mu_E$]
\label{thm:muE_zero_lipschitz}
Let $f\colon\mathcal{C}\to\mathcal{C}$ and let $\pi$ be an ordering. The following are
equivalent.
\begin{enumerate}[label=\textup{(\roman*)}]
  \item $\mu_E(\pi)=0$.
  \item\label{locus:multipliers} $\Lambda_{n,m}\le 1$ for every $0\le n\le N-1$ and every $m$, that is, no
        observable ball is expanding at any level.
  \item Every prefix $J^\pi_n$, $1\le n\le N-1$, is autonomous.
  \item $f$ is triangular for $\pi$.
  \item $F_\pi$ is compatible.
  \item $\widehat{F}_\pi$ is $1$-Lipschitz.
\end{enumerate}
\end{theorem}

\begin{proof}
(i)~$\Leftrightarrow$~(ii). By~\eqref{eq:EAI_sets} the set $E(n)$ collects the indices
with $\Lambda_{n,m}>1$, and by~\eqref{eq:mu_definition} the score $\mu_E(\pi)$ is a sum
of nonnegative terms with strictly positive weights $p^{N-n}$. It vanishes if and only
if every $E(n)$ is empty, which is (ii).

(i)~$\Leftrightarrow$~(iii). Same argument applied
to~\eqref{eq:mu_master_formula}: all the counts are nonnegative integers and all the
weights are positive, so the sum vanishes if and only if every count vanishes, that is,
if and only if no prefix fails autonomy at any $u$.

(iii)~$\Leftrightarrow$~(v). Fix $n$. The condition in
Definition~\ref{def:triangular_compatible}(b) at level $n$ says that
$\operatorname{trunc}_n\circ F_\pi$ factors through $\operatorname{trunc}_n$, that is,
that all the image centers of a level-$n$ ball agree modulo $p^{n}$, for every ball. It
is vacuous at $n=0$, where $\operatorname{trunc}_0$ is constant, and vacuous at $n=N$,
where $\operatorname{trunc}_N$ is the identity of $\mathbb{Z}/p^{N}\mathbb{Z}$ so that
the hypothesis $x\equiv y\pmod{p^{N}}$ already forces $x=y$. For $1\le n\le N-1$ it is
the autonomy of $J^\pi_n$ at every $u$, by~\eqref{eq:E_is_nonautonomy}. Ranging over
$n$ gives the equivalence.

(v)~$\Rightarrow$~(iv). Fix $0\le i\le N-1$ and let $a,b\in\mathcal{C}$ agree in the
coordinates $\pi(0),\dots,\pi(i)$. Then $m_a\equiv m_b\pmod{p^{\,i+1}}$, so (v) applied
with $n=i+1$ gives $m_{f(a)}\equiv m_{f(b)}\pmod{p^{\,i+1}}$. Comparing the digit in
position $i$ on both sides yields $f(a)_{\pi(i)}=f(b)_{\pi(i)}$. Thus $f(a)_{\pi(i)}$ is
determined by $a_{\pi(0)},\dots,a_{\pi(i)}$, which is (iv).

(iv)~$\Rightarrow$~(v). Let $x\equiv y\pmod{p^{n}}$ with $0\le n\le N$, and write
$x=m_a$, $y=m_b$. Then $a$ and $b$ agree in the coordinates
$\pi(0),\dots,\pi(n-1)$, so by (iv) the values $f(a)_{\pi(i)}$ and $f(b)_{\pi(i)}$
coincide for every $i\le n-1$. These are the digits in positions $0,\dots,n-1$ of
$m_{f(a)}$ and $m_{f(b)}$, whence $F_\pi(x)\equiv F_\pi(y)\pmod{p^{n}}$.

(v)~$\Leftrightarrow$~(vi). Let $z,w\in\mathbb{Z}_p$. If $|z-w|_p\le p^{-N}$ then
$\operatorname{trunc}_N(z)=\operatorname{trunc}_N(w)$, so
$\widehat{F}_\pi(z)=\widehat{F}_\pi(w)$ and the Lipschitz inequality holds trivially at
that scale and at every finer one. If $|z-w|_p=p^{-n}$ with $0\le n\le N-1$ then
$\operatorname{trunc}_n(z)=\operatorname{trunc}_n(w)$, and the inequality
$|\widehat{F}_\pi(z)-\widehat{F}_\pi(w)|_p\le p^{-n}$ says exactly that the two image
centers are congruent modulo $p^{n}$. Ranging over $n$, the $1$-Lipschitz condition for
$\widehat{F}_\pi$ is the conjunction of the congruences in (v) for
$0\le n\le N-1$, and the remaining case $n=N$ of (v) is vacuous as noted above. This completes the proof.
\end{proof}

\begin{remark}[$\mu_E$ as a graded obstruction to being an automaton]
\label{rem:muE_graded_obstruction}
Conditions (ii), (iv), (v) and (vi) are a standard package in the $p$-adic theory of
automata. By~\cite[Theorem~2.10]{anashin2021automata_chapter} the $1$-Lipschitz maps of
$\mathbb{Z}_p$ are exactly the automaton functions of initial transducers on the alphabet
$\mathbb{F}_p$, whose state set is not required to be
finite~\cite[\S2.3]{anashin2021automata_chapter} and where the finite-state ones form a
proper subclass treated separately there, by~\cite[Corollary~2.12]{anashin2021automata_chapter} exactly the maps
compatible with every congruence modulo $p^{r}$, and
by~\cite[Proposition~2.13]{anashin2021automata_chapter} and \cite[Proposition~3.35]{Anashin2009} exactly the triangular ones. The
same vocabulary, distributed over several statements, is in~\cite[\S3.8]{Anashin2009}.
Item~\ref{locus:multipliers} is the register in which that literature works: a map is an automaton function
if and only if it sends every ball of radius $p^{-r}$ into a ball of radius $p^{-r}$ for
$r\ge1$~\cite[Corollary~2.11]{anashin2021automata_chapter}, which for $r=n\le N-1$ is
$M_{n,m}\ge n$, that is $\Lambda_{n,m}\le1$, the scales $r\ge N$ being automatic because
$\widehat{F}_\pi$ is constant on balls of radius $p^{-N}$. So $\mu_E(\pi)=0$ says that the
finite system, read in the ordering $\pi$, is a $p$-adic automaton in that transducer sense, and
formula~\eqref{eq:mu_master_formula} exhibits $\mu_E$ as the graded obstruction to being
one: it records which prefixes fail to be autonomous and at which resolution. The approximants
carry the same reading: by Definition~\ref{def:nth_approximation} the image of
$\BCP_{1/p^n}(m)$ under $F_n$ is a ball of radius $t_{n,m}$, so $\Lambda_{n,m}\le1$ says
exactly that $F_n$ sends that ball to one of no larger radius.
\end{remark}

\noindent The triangular maps can be counted exactly, and the count shows that the vanishing locus is a vanishingly small part of the $p^{\,Np^{N}}$ maps of $\mathcal{C}$, the same part for every ordering. The Supplementary Material records the computation. What matters here is which orderings reach the locus, and the dependency digraph answers that.

\begin{definition}[Essential dependence and the dependency digraph]
\label{def:dependency_digraph}
The coordinate function $f(\cdot)_i$ \emph{depends essentially} on $a_j$ if there exist
$a,b\in\mathcal{C}$ differing only in the coordinate $j$ with $f(a)_i\ne f(b)_i$. The
\emph{dependency digraph} $G_f$ has vertex set $\{0,\dots,N-1\}$ and an edge $j\to i$
whenever $f(\cdot)_i$ depends essentially on $a_j$. Loops are allowed and record
self-dependence. This is the graph drawn for the toy model in the Supplementary
Material.
\end{definition}

\begin{corollary}[When the minimum is zero]
\label{cor:min_zero_dag}
Let $G_f^{\circ}$ denote $G_f$ with its loops deleted. Then
\[
\min_{\pi\in S_N}\mu_E(\pi)=0\iff G_f^{\circ}\ \text{is acyclic},
\]
and in that case $\{\pi:\mu_E(\pi)=0\}$ is exactly the set of topological orders of
$G_f^{\circ}$, that is, the orderings $\pi$ for which every edge $j\to i$ of
$G_f^{\circ}$ satisfies $\pi^{-1}(j)<\pi^{-1}(i)$.
\end{corollary}

\begin{proof}
We first record the elementary fact that a function $h\colon\mathcal{C}\to\mathbb{F}_p$
is determined by the coordinates in a set $T$ if and only if it depends essentially on
no coordinate outside $T$. If $h$ is determined by the coordinates in $T$ and $j\notin T$,
then two configurations differing only at $j$ agree on $T$ and therefore receive the same
value under $h$, so $h$ does not depend essentially on $j$. For the converse, suppose $h$
depends essentially on no $j\notin T$ and let $a,b$ agree on $T$. Changing the
coordinates outside $T$ one at a time turns $a$ into $b$ through a chain of
configurations, consecutive members of which differ in a single coordinate outside $T$,
so $h$ is unchanged along the chain and $h(a)=h(b)$.

By Theorem~\ref{thm:muE_zero_lipschitz}(iv), $\mu_E(\pi)=0$ holds if and only if, for
every $i$, the function $f(\cdot)_{\pi(i)}$ is determined by the coordinates
$\pi(0),\dots,\pi(i)$, which by the previous paragraph is the same as saying that
$f(\cdot)_{\pi(i)}$ depends essentially on no coordinate placed after position $i$. In
terms of $G_f$ this says that every edge $j\to k$ with $j\ne k$ satisfies
$\pi^{-1}(j)<\pi^{-1}(k)$, the inequality being strict because two distinct coordinates
occupy distinct positions, while a loop $k\to k$ imposes the vacuous condition
$\pi^{-1}(k)\le\pi^{-1}(k)$. That is the definition of a topological order of
$G_f^{\circ}$. Finally, a finite digraph admits a topological order if and only if it is
acyclic. Along a directed cycle the positions $\pi^{-1}$ would increase strictly and
return to their starting value, so no topological order exists when a cycle is present.
Conversely, a digraph with no cycle has a vertex of in-degree zero, since a walk taken
backwards along edges visits at most $N$ distinct vertices, and a walk that never stops
repeats one and closes a cycle. Placing such a vertex first and repeating the argument on
the digraph that remains builds a topological order in $N$ steps. An ordering with
$\mu_E(\pi)=0$ therefore exists if and only if $G_f^{\circ}$ is acyclic, and since
$\mu_E$ is a nonnegative function on the finite nonempty set $S_N$, its minimum is
attained and vanishes exactly when some ordering attains zero. That is the asserted
equivalence, the identification of $\{\pi:\mu_E(\pi)=0\}$ with the set of topological
orders having been established above.
\end{proof}

In the vocabulary of the application, $\mu_E$ reaches $0$ exactly on systems with no
feedback between distinct coordinates, and the minimizers are then the hierarchies that
read masters before subordinates. The toy model of Section~\ref{sec:toy_example} is such
a system: its updates~\eqref{eq:toy_polys} are triangular in the displayed ordering,
which is why $\mu_E=0$ there, and its loop-free dependency digraph has the two
topological orders $(x_0,x_1,x_2,x_3)$ and $(x_1,x_0,x_2,x_3)$, which by
Corollary~\ref{cor:min_zero_dag} are exactly the minimizers. For systems with feedback,
and the \textit{A.\ thaliana} network of Section~\ref{sec:athaliana} is one, the
positive value $\min_\pi\mu_E>0$ is the obstruction to triangularizing $f$ by permuting
coordinates, measured across scales.

For each level $n$ and each type $T\in\{A,E,I\}$, let
\begin{equation}
\label{eq:Un_expanding_region}
U_n^{T}:=\bigsqcup_{m\in T(n)}B_{n,m}\ \subset\ \mathbb{Z}_p
\end{equation}
be the union of the observable source cylinders carrying that label, and write $\Phi^{T}_\pi(x)$ for the number of levels $0\le n\le N-1$ at which the ball containing $x$ is of type $T$. Haar measure is applied to these source cylinders. It is not applied to their analytic image balls in $\mathbb{C}_p$, nor to the pointwise sets $F_n(B_{n,m})$. Since the $B_{n,m}$ are disjoint of measure $p^{-n}$, one has $\lambda_{\mathrm{H}}(U_n^{T})=|T(n)|p^{-n}$, and~\eqref{eq:weight_is_haar} turns each of the three scores into an integral.

\begin{proposition}[Haar integral form]
\label{prop:haar_integral_form}
Let $\lambda_{\mathrm{H}}$ be the normalized Haar measure on $\mathbb{Z}_p$. For every ordering $\pi$ and every $T\in\{A,E,I\}$,
\begin{equation}
\label{eq:mu_haar_integral}
\mu_T(\pi)=p^N\!\int_{\mathbb{Z}_p}\Phi^{T}_\pi(x)\,d\lambda_{\mathrm{H}}(x),
\qquad\text{and}\qquad
\Phi^{A}_\pi+\Phi^{E}_\pi+\Phi^{I}_\pi\equiv N .
\end{equation}
\end{proposition}
\begin{proof}
By~\eqref{eq:Un_expanding_region} we may write $\Phi^{T}_\pi=\sum_{n=0}^{N-1}\mathbf{1}_{U_n^{T}}$. Each $U_n^{T}$ is clopen. Indeed, $B_{n,m}$ is the residue class $m+p^{n}\mathbb{Z}_p$ of Section~\ref{sec:setup}, which is open because the value group $|\mathbb{Q}_p^{\times}|=p^{\mathbb{Z}}$ is discrete, so that the closed ball of radius $p^{-n}$ about $m$ coincides with the open ball of radius $p^{-n+1}$, and closed because its complement in $\mathbb{Z}_p$ is the union of the other $p^{n}-1$ classes at that level. A finite union of such classes is again clopen. Hence $\Phi^{T}_\pi$ is a finite sum of indicator functions of clopen sets, locally constant and integrable, and
\[
\int_{\mathbb{Z}_p}\Phi^{T}_\pi\,d\lambda_{\mathrm{H}}=\sum_{n=0}^{N-1}\lambda_{\mathrm{H}}(U_n^{T})=\sum_{n=0}^{N-1}|T(n)|p^{-n}.
\]
Multiplying by $p^N$ gives $\sum_{n=0}^{N-1}|T(n)|p^{N-n}$, which is $\mu_T(\pi)$ by~\eqref{eq:mu_definition}. For the second identity, each $x$ lies in exactly one ball at each of the $N$ levels $n=0,\dots,N-1$, and that ball carries exactly one of the three labels, so the three counts partition $\{0,\dots,N-1\}$ pointwise. Integrating it recovers Lemma~\ref{lem:mu_checksum}.
\end{proof}

\noindent We write $\Phi_\pi:=\Phi^{E}_\pi$ from here on, since the expansion score is the one that is minimized. The identity $\Phi^{A}_\pi+\Phi^{E}_\pi+\Phi^{I}_\pi\equiv N$ refines the control identity pointwise: the budget $Np^N$ is spent branch by branch, $N$ units on each.

\noindent The sum counts with multiplicity across resolutions. A single point may lie in expanding cylinders at several levels, and each of them contributes, so the two sides of
\[
\lambda_{\mathrm{H}}\Bigl(\bigcup_{n}U^{E}_n\Bigr)\ \le\ \sum_{n}\lambda_{\mathrm{H}}(U^{E}_n)
\]
agree only when the $U^{E}_n$ are pairwise disjoint. The inequality is the subadditivity of $\lambda_{\mathrm{H}}$, and the criterion for equality is exact here because a nonempty intersection of two of these clopen sets already has positive measure.
Accordingly $\mu_E(\pi)/p^N$ is the Haar average number of expanding resolutions along a branch of the residue tree. It is an average of a function with values in $\{0,1,\dots,N-1\}$, so it lies in the interval $[0,N-1]$, and inside that interval it can only land in the finite set $p^{-N}\mathbb{Z}\cap[0,N-1]$, because $\mu_E(\pi)$ is a nonnegative integer. Proposition~\ref{prop:Phi_pointwise} identifies the branches it averages over.

\begin{proposition}[Pointwise reading of the Haar integrand]
\label{prop:Phi_pointwise}
For $x\in\mathbb{Z}_p$,
\[
\Phi_\pi(x)=\#\Bigl\{0\le n\le N-1\ :\ \exists\,y\in\mathbb{Z}_p \text{ with }
|y-x|_p\le p^{-n} \text{ and }
\bigl|\widehat{F}_\pi(y)-\widehat{F}_\pi(x)\bigr|_p>p^{-n}\Bigr\},
\]
so that $\Phi_\pi(x)$ counts the scales at which the $1$-Lipschitz inequality fails at
the point $x$, and $\mu_E(\pi)/p^{N}$ is the average number of such scales, a number in
$[0,N-1]$.
\end{proposition}

\begin{proof}
Let $B_{n,m}$ be the level-$n$ ball containing $x$. The map $\widehat{F}_\pi$ is constant
on balls of radius $p^{-N}$, so its values on $B_{n,m}$ are exactly the centers
$m_{f(a)}$ with $\operatorname{trunc}_n(a)=m$, and two of them satisfy
$|\widehat{F}_\pi(y)-\widehat{F}_\pi(z)|_p>p^{-n}$ if and only if they fail to be
congruent modulo $p^{n}$. By Lemma~\ref{lem:initial_segment} the ball is therefore
expanding, that is $M_{n,m}<n$, exactly when such a pair $y,z\in B_{n,m}$ exists. Suppose
it does. By the strong
triangle inequality
\[
\bigl|\widehat{F}_\pi(y)-\widehat{F}_\pi(z)\bigr|_p\le\max\Bigl\{
\bigl|\widehat{F}_\pi(y)-\widehat{F}_\pi(x)\bigr|_p,\
\bigl|\widehat{F}_\pi(x)-\widehat{F}_\pi(z)\bigr|_p\Bigr\},
\]
so at least one of the two terms on the right exceeds $p^{-n}$, and the corresponding
point lies in $B_{n,m}$, hence at distance at most $p^{-n}$ from $x$. This produces the
witness $y$ required by the displayed set. Conversely such a witness lies in $B_{n,m}$
and shows that the ball is expanding. The two descriptions of $\Phi_\pi$ therefore agree
level by level, and the statement about the average is
Proposition~\ref{prop:haar_integral_form} divided by $p^{N}$, the bound $N-1$ coming
from $E(0)=\emptyset$.
\end{proof}

\begin{theorem}[The saturation locus of $\mu_E$]
\label{thm:muE_saturation}
Let $f\colon\mathcal{C}\to\mathcal{C}$.
\begin{enumerate}[label=\textup{(\alph*)}]
  \item For every ordering $\pi$ one has $\mu_E(\pi)\le (N-1)p^{N}$, with equality if
        and only if no prefix $J^\pi_n$, $1\le n\le N-1$, is autonomous at any
        $u\in\mathbb{F}_p^{J^\pi_n}$.
  \item $\mu_E(\pi)=(N-1)p^{N}$ for every $\pi\in S_N$, that is, $f$ is saturated in the
        sense of Remark~\ref{rem:saturated}, if and only if no set of coordinates $J$
        with $1\le|J|\le N-1$ is autonomous at any $u\in\mathbb{F}_p^{J}$.
  \item $\mu_E(\pi)=0$ for every $\pi\in S_N$ if and only if $f$ acts coordinatewise,
        that is, $f(a)_i=g_i(a_i)$ for maps $g_i\colon\mathbb{F}_p\to\mathbb{F}_p$.
\end{enumerate}
\end{theorem}

\begin{proof}
(a) In~\eqref{eq:mu_master_formula} the level-$n$ count is bounded by the number of
assignments $u$, namely $p^{n}$, so
$\mu_E(\pi)\le\sum_{n=1}^{N-1}p^{n}p^{\,N-n}=(N-1)p^{N}$. Since all the terms are
nonnegative and the weights are positive, equality holds if and only if every count is
maximal, which is the stated condition.

For (b) and (c) we use the following two observations. First, as $\pi$ runs over $S_N$
and $n$ over $\{1,\dots,N-1\}$, the prefixes $J^\pi_n$ run over all subsets $J$ of
$\{0,\dots,N-1\}$ with $1\le|J|\le N-1$: any ordering that lists the elements of a given
such $J$ in its first $|J|$ positions realizes $J$ as $J^\pi_{|J|}$, and conversely
every prefix has size between $1$ and $N-1$. Second, autonomy of $J$ at $u$ depends only
on the pair $(J,u)$ and not on the ordering that produced $J$, because both the fiber
$\mathcal{C}(J,u)$ and the restriction $f(a)|_{J}$ are defined without reference to any
order on $J$.

(b) By (a), the equality $\mu_E(\pi)=(N-1)p^{N}$ holds for all $\pi$ if and only if no
prefix of any ordering is autonomous at any $u$, which by the two observations is the
same as saying that no $J$ with $1\le|J|\le N-1$ is autonomous at any $u$.

(c) By Theorem~\ref{thm:muE_zero_lipschitz}(iii), the equality $\mu_E(\pi)=0$ holds for
all $\pi$ if and only if every prefix of every ordering is autonomous, which by the two
observations is the same as saying that every $J$ with $1\le|J|\le N-1$ is autonomous.
If this holds and $N\ge 2$, take $J=\{i\}$: autonomy of a singleton says exactly that
$f(a)_i$ is a function of $a_i$ alone, which is the stated form. Conversely, if
$f(a)_i=g_i(a_i)$ for every $i$, then $f(a)|_{J}$ is determined by $a|_{J}$ for every
$J$, so every $J$ is autonomous. For $N=1$ both sides of (c) hold vacuously, since
$\mu_E\equiv 0$ and every map of $\mathbb{F}_p$ is of the stated form.
\end{proof}

The saturation locus is not empty, and among a familiar family it is describable in closed form. The statement below labels the cells of the ring by $\mathbb{Z}/N\mathbb{Z}$, and we read that label set as the coordinate set $\{0,\dots,N-1\}$ of Definition~\ref{def:autonomous_block} through least nonnegative residues. Nothing depends on the choice: the criterion of Theorem~\ref{thm:muE_saturation}(b) quantifies over all coordinate subsets $J$ with $1\le|J|\le N-1$ and imposes no order on them, so it is insensitive to which bijection between the two label sets is used, and the cyclic neighborhood $\{j-1,j,j+1\}$ is a statement about $f$ and not about the ordering $\pi$. The statement is proved in the Supplementary Material.

\begin{theorem}[Saturated elementary cellular automata]
\label{thm:eca_saturation_main}
Let $p=2$, index the cells of a ring by $\mathbb{Z}/N\mathbb{Z}$ with $N\ge5$, and let $f(a)_j=r(a_{j-1},a_j,a_{j+1})$ be the elementary cellular automaton with local rule $r$ and periodic boundary. Then $f$ is saturated in the sense of Theorem~\ref{thm:muE_saturation}(b) if and only if $r$ is permutive in an outer variable, that is, if either $x\mapsto r(x,y,z)$ is a bijection for every $(y,z)$ or $z\mapsto r(x,y,z)$ is a bijection for every $(x,y)$. Exactly $28$ of the $256$ elementary rules qualify.
\end{theorem}

\noindent Permutivity in an outer variable is Hedlund's condition on the local rule of a shift endomorphism~\cite{hedlund1969endomorphisms}. Saturation is therefore a recognizable structural property in that family and not a pathology: the ordering carries no information precisely when the local rule is a bijection in one of the two directions along which information travels.

\begin{proposition}[Order-independent profiles and a sharp intermediate regime]
\label{prop:intermediate_constant}
For a nonempty proper $J\subsetneq\{0,\dots,N-1\}$ put
\[
q_f(J):=\#\bigl\{u\in\mathbb{F}_p^{J}\ :\ J\text{ is not autonomous at }u\bigr\}.
\]
\begin{enumerate}[label=\textup{(\alph*)}]
\item $\mu_E$ is constant on $S_N$ if and only if $q_f(J)$ depends only on $|J|$. In that
case there are integers $0\le\kappa_n\le p^{n}$, $1\le n\le N-1$, with $q_f(J)=\kappa_n$
whenever $|J|=n$, and $\mu_E(\pi)=\sum_{n=1}^{N-1}\kappa_n p^{\,N-n}$ for every $\pi$.
\item If every nonempty proper $J$ fails autonomy at at least one assignment, then
$\mu_E(\pi)\ge(p^{N}-p)/(p-1)$ for every $\pi$, and this bound is attained, for every
prime $p$ and every $N\ge2$, by
\[
f(a)_i:=a_i+\delta(a),\qquad
\delta(a):=\prod_{r=0}^{N-1}\bigl(1-(a_r+1)^{p-1}\bigr),
\]
for which $\kappa_n=1$ at every level. At $p=2$ this is $f(a)_i=a_i+a_0a_1\cdots a_{N-1}$.
\end{enumerate}
\end{proposition}

\begin{proof}
(a) Suppose $\mu_E$ is constant. Fix $S$ with $|S|=n-1$ and $i,j\notin S$ distinct, and
let $\pi,\pi'$ agree except for exchanging $i$ and $j$ at the consecutive positions $n-1$
and $n$. Their prefix chains agree at every level but the $n$-th, where they are
$S\cup\{i\}$ and $S\cup\{j\}$, so~\eqref{eq:mu_master_formula} gives
\[
0=\mu_E(\pi)-\mu_E(\pi')=p^{\,N-n}\bigl(q_f(S\cup\{i\})-q_f(S\cup\{j\})\bigr).
\]
Two $n$-sets differing in one element therefore carry the same value, and the Johnson
graph on the $n$-sets is connected, so $q_f$ is constant on each level. The converse is
immediate from~\eqref{eq:mu_master_formula}. Adjacent transpositions isolate a single
level, which is what rules out compensation between levels.

(b) Under the hypothesis every level contributes at least one assignment, so
$\mu_E(\pi)\ge\sum_{n=1}^{N-1}p^{\,N-n}=(p^{N}-p)/(p-1)$. The sum starts at $n=1$ because
the empty prefix is autonomous for trivial reasons, the same fact that empties $E(0)$ and
lowers the ceiling of Theorem~\ref{thm:muE_saturation}(a) from $Np^N$ to $(N-1)p^N$. For
the displayed map, $(a_r+1)^{p-1}$ is $1$ unless $a_r=p-1$, so $\delta$ is the indicator
of the state all of whose coordinates equal $p-1$. On a fiber $\mathcal{C}(J,u)$ one has
$f(a)|_J=(u_i+\delta(a))_{i\in J}$, and since $J\neq\varnothing$ this tuple determines and
is determined by $\delta(a)$. If some $u_i\neq p-1$ then $\delta$ vanishes on the whole
fiber and $J$ is autonomous at $u$. If every $u_i=p-1$ then the fiber contains the state
of all $p-1$, where $\delta=1$, and others where it vanishes, because $J$ is proper. So
$q_f(J)=1$ for every proper $J$, giving $\kappa_n=1$ and
$\mu_E(\pi)=(p^{N}-p)/(p-1)$ for every $\pi$.
\end{proof}

\begin{remark}[Constancy of $\mu_E$ across $S_N$ is a diagnostic only at the maximum]
\label{rem:constancy_diagnostic}
Theorem~\ref{thm:muE_saturation}(b) and (c) describe two constant regimes, at the maximum
and at the minimum, and they are different phenomena: constancy at $(N-1)p^{N}$ says that
no proper set of coordinates ever determines its own next state, constancy at $0$ that the
coordinates act independently. Constancy alone separates neither case, since by
Proposition~\ref{prop:intermediate_constant} the score can be constant on all of $S_N$ at
an intermediate value. The reason is that the map of that proposition commutes with every
permutation of the coordinates, and by Theorem~\ref{lem:balls_are_blocks} such a symmetry
forces the level-$n$ count to depend on $n$ alone. What is a diagnostic is constancy
\emph{at the maximum}.
\end{remark}

\noindent Read across the hierarchy, the $A/E/I$ labels assemble into a discrete, scale-resolved ultrametric landscape on $\mathbb{Z}_p$, whose relief is the ball-image geometry measured by $\Lambda_{n,m}$ and moves with the ordering. The Supplementary Material renders it and tabulates the correspondence between ball type and local $p$-adic dynamics.

%% file: 07_application_athaliana.tex
\section{Worked example: \texorpdfstring{\textit{Arabidopsis thaliana}}{Arabidopsis thaliana} floral GRN}
\label{sec:athaliana}

The variational problem of Section~\ref{sec:stability_measure} is a minimization over $S_N$, and two questions about it are settled only by running it on a map that was not built for the purpose: whether the minimum can be certified at a size where $N!$ is large, and whether $\mu_E$ separates orderings by enough to be worth minimizing. We take the floral development network of \textit{A.\ thaliana}, a Boolean model with $N=13$ genes whose transition table is fixed in~\cite{mendoza1998dynamics,espinosa2004gene}. We use the standard $13$-variable reduction of the published $15$-element network, obtained by fixing the constitutively expressed inputs $\mathrm{LUG}=\mathrm{CLF}=1$ and deleting them; the Boolean form is the one developed in~\cite{perez2022epigenetic} from~\cite{espinosa2004gene}, the deposited table was validated exhaustively against an independent transcription of those rules, and the encoding order and the full table are archived in~\cite{padic_grn_bundle}. The later variant of the same network with minimized rules~\cite{chaos2006genes} is a different map, differing from this one on $2{,}600$ of the $8{,}192$ transitions while sharing all ten fixed points. Running the minimization answers both questions. The biological reading of the output, the network diagram, and the regulator table are in the Supplementary Material.

The map has $\mathcal{C}=\mathbb{F}_2^{13}$, so $|\mathcal{C}|=8192$ and $|S_{13}|=13!$. Exactly four orderings attain the global minimum of $\mu_E$, certified by branch and bound over all of $S_{13}$, and cross-checked by a second exact algorithm that does not enumerate orderings, with both runs in the reproducibility bundle~\cite{padic_grn_bundle}. We use the representative
\begin{equation}
\label{eq:athaliana_optimal_ordering}
\begin{aligned}
\text{position } k:\quad (0,\ 1,\ 2,\ 3,\ 4,\ 5,\ &6,\ 7,\ 8,\ 9,\ 10,\ 11,\ 12)\\
    \|&\\
\text{gene } \pi^*(k):\quad (\text{UFO, EMF1, LFY, TFL1, AP2, FT}, &\text{AG, AP1, SEP, AP3, PI, WUS, FUL}).
\end{aligned}
\end{equation}
Under $\pi^*$ two configurations are close in the $2$-adic metric when they agree on the early genes. The score triple is
\[
(\mu_E,\mu_A,\mu_I)(\pi^*)=(26{,}776,\,53{,}880,\,25{,}840),
\qquad \mu_E+\mu_A+\mu_I=13\cdot 2^{13},
\]
in agreement with Lemma~\ref{lem:mu_checksum}. The ordering in which the network is presented in~\cite{mendoza1998dynamics,espinosa2004gene} gives $\mu_E=80{,}704$ on the same map, $3.014$ times the optimum and $82.1\%$ of the saturation value $(N-1)2^N=98{,}304$. The functional therefore ranges widely over $S_{13}$, and the ordering it selects is a genuine choice among the $13!$ available.

\begin{remark}[The $\mu_E$-optimal set is four-fold degenerate]
\label{rem:symmetry_C2C2}
The global minimum $\mu_E = 26{,}776$ is attained by exactly four orderings. They differ from $\pi^*$ by transposing $\mathrm{EMF1}$ with $\mathrm{LFY}$, which occupy positions $1$ and $2$, and $\mathrm{AP2}$ with $\mathrm{FT}$, at positions $4$ and $5$, positions being numbered from $0$ as in~\eqref{eq:athaliana_optimal_ordering}. All four assign the same $A/E/I$ letter to every configuration at every resolution, and therefore share the full triple $(\mu_E,\mu_A,\mu_I)$ and the level-wise counts $|E(n)|$, $|A(n)|$, $|I(n)|$, verified computationally. Their prefix partitions are not the same: at levels $2$ and $5$, exactly where the two exchangeable pairs enter, the four orderings group the configurations differently. The degeneracy is therefore invisible to the scores and to the labelled $A/E/I$ history, though not to the coordinate-labelled hierarchy itself. The four orderings form a single orbit of the group $\langle\tau_{12}\rangle\times\langle\tau_{45}\rangle\cong C_2\times C_2$ generated by those two transpositions, so $\mathrm{Opt}(f)$ is one coset of that group rather than a single point.
\end{remark}

\paragraph{Fixed points and ball-level types.}
\label{subsec:athaliana_attractors}

The map has exactly ten fixed points, and under $\pi^*$ they fall into four balls at resolution $n=4$. Configuration indices use the encoding of~\eqref{eq:mapi_definition}, $m_a=\sum_{k=0}^{N-1}a_{\pi^*(k)}2^{\,k}$, so that digit $k$ carries the gene $\pi^*(k)$ placed at position $k$ by~\eqref{eq:athaliana_optimal_ordering}; the ball assignment moves with the ordering.

\begin{itemize}
    \item Contracting $B_{1/16}(4)$: four fixed points, $436$, $1972$, $5492$, $6004$.
    \item Contracting $B_{1/16}(5)$: two fixed points, $1973$ and $6005$.
    \item Expanding $B_{1/16}(10)$: two fixed points, $10$ and $2058$.
    \item Expanding $B_{1/16}(11)$: two fixed points, $11$ and $2059$.
\end{itemize}

Six of the ten configurations therefore lie in level-$4$ contracting balls and four in level-$4$ expanding balls, a distinction the fixed-point equation does not make: as configurations with $f(a)=a$ they are indistinguishable. The words contracting and expanding describe the level-$4$ image-radius geometry realized by $F_4$. They do not classify the fixed configurations of $f$ as attractors or repellers of its functional graph, whose combinatorial basins the labels do not measure. Figure~\ref{fig:hierarchical_padic_athaliana} shows the cascade with these four balls marked, and the chains themselves are not uniformly of one type, as the words of Table~\ref{tab:fp_sequences} show.

\begin{sidewaysfigure}
    \centering
    \includegraphics[width=0.95\textheight]{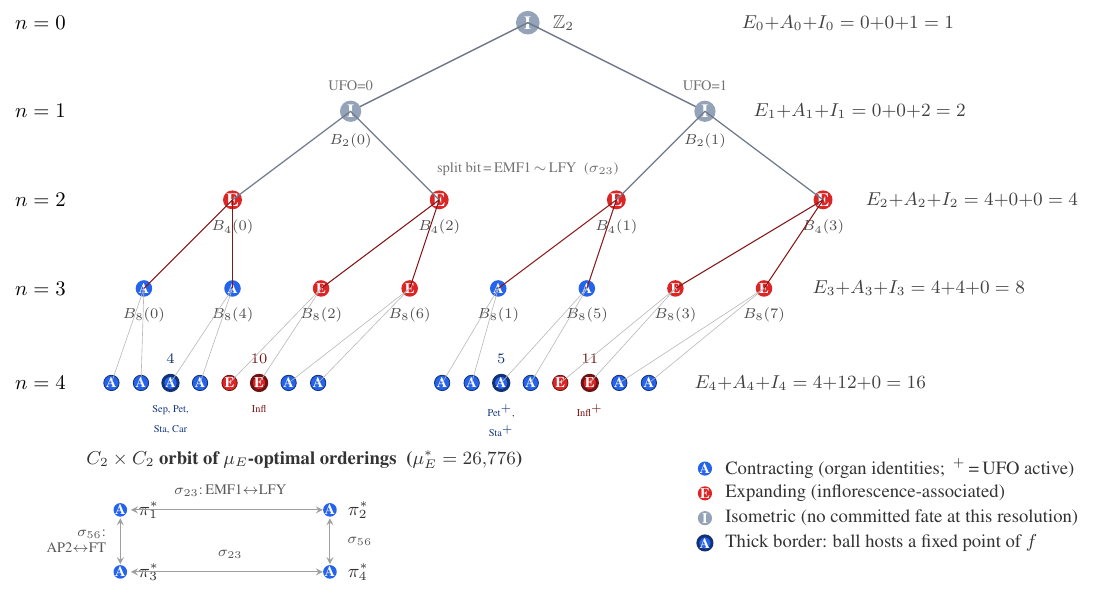}
    \caption{Hierarchical $2$-adic ball cascade for the \textit{A.\ thaliana} floral GRN under $\pi^*$, drawn to resolution $n=4$. The ten fixed points occupy four balls at that resolution, drawn with a thick border: the two contracting ones hold six of them and read IEAA, the two expanding ones hold the remaining four and read IEEE (Table~\ref{tab:fp_sequences}). The cascade drawn is the one determined by $\pi^*$. The inset shows the $C_2\times C_2$ orbit of the four $\mu_E$-optimal orderings of Remark~\ref{rem:symmetry_C2C2}, which agree with it on the type of every configuration at every resolution while grouping the configurations differently at level~$2$.}
    \label{fig:hierarchical_padic_athaliana}
\end{sidewaysfigure}

\paragraph{Scale-resolved classification.}

For each fixed point we record its $A/E/I$ word (Definition~\ref{def:aei_word}) across the first four resolutions, the prefix of the full word of length $N-1=12$. These are the scales at which the ten points separate.

\begin{table}[htbp]
\centering
\caption{Classification sequence of the ten fixed points under $\pi^*$. Seq.\ = classification at levels $n=1,2,3,4$, the four-scale prefix of the full length-$(N{-}1)$ word. Level $n=1$ is I for all balls, since each level-1 ball maps to itself. A=contracting, E=expanding, I=isometric, computed from $\Lambda_{n,m}$. The state column records the identification made in~\cite{mendoza1998dynamics,espinosa2004gene}.}
\label{tab:fp_sequences}
\footnotesize
\begin{tabular}{@{}rclll@{}}
\toprule
Config & Ball & Seq. & Ball type & State \\
\midrule
436 & $B_{1/16}(4)$ & IEAA & Contracting & Sepal \\
1972 & $B_{1/16}(4)$ & IEAA & Contracting & Petal \\
5492 & $B_{1/16}(4)$ & IEAA & Contracting & Carpel \\
6004 & $B_{1/16}(4)$ & IEAA & Contracting & Stamen \\
1973 & $B_{1/16}(5)$ & IEAA & Contracting & Petal$+$UFO \\
6005 & $B_{1/16}(5)$ & IEAA & Contracting & Stamen$+$UFO \\
\midrule
10 & $B_{1/16}(10)$ & IEEE & Expanding & Inflorescence \\
2058 & $B_{1/16}(10)$ & IEEE & Expanding & Inflorescence \\
11 & $B_{1/16}(11)$ & IEEE & Expanding & Inflorescence \\
2059 & $B_{1/16}(11)$ & IEEE & Expanding & Inflorescence \\
\bottomrule
\end{tabular}
\end{table}

The word splits the ten points cleanly in two. Both classes expand at level $2$. In one, the expansion gives way to contraction at levels $3$ and $4$, the ball-level multiplier crossing $1$, while in the other it persists at every scale. A label read at one resolution records a single cross-section of that transition, while the word records where the crossover falls and whether it persists. The normalized pointwise one-bit sensitivity $s_f(a)=N^{-1}\sum_{i}\mathbf{1}[f(a\oplus e_i)\neq f(a)]$, the summand of the standard average sensitivity~\cite{kadelka2024metaanalysis}, which on a fixed point reads $\mathbf{1}[f(a\oplus e_i)\neq a]$ since $f(a)=a$, takes the value $7/13$ on all four points of the second class, and that value lies inside the range $[6/13,\,11/13]$ spanned by the first: one point of the first class is strictly less sensitive than every point of the second, and another matches them exactly, so no threshold on $s_f$ separates the two. What separates them is the resolution-by-resolution profile, an object of the ball tree.

\paragraph{What the ordering finds on its own.}

The four coarsest positions of $\pi^*$ are occupied by UFO, EMF1, LFY and TFL1. UFO is the one coordinate the dynamics conserves, $\mathrm{UFO}'=\mathrm{UFO}$, so it is the unique source of the dependency digraph once its loops are deleted (Definition~\ref{def:dependency_digraph}) and the unique coordinate whose singleton prefix is autonomous, and the minimization sends it to the one level where an autonomous prefix costs nothing. EMF1, LFY and TFL1 are documented regulators of the floral transition~\cite{mendoza1998dynamics,espinosa2004gene,chavez2022flowering}, and none of them is a source, so their positions at levels two to four are decided by the functional and not by the shape of the digraph. UFO has a documented role of its own, in organ identity, recorded with the rest in the Supplementary Material. A study of this and other regulatory networks, with mutant perturbations and cross-organism comparisons, appears in~\cite{perez2026padic_applications}, with its own aims and methods, and a complementary topological analysis of the same network, without $p$-adic methods, is in~\cite{cortes2026forest}. No gene label, expression measurement, or prior ranking enters the optimization, which sees the transition table and the encoding rule and nothing else.

%% file: 08_discussion.tex
\section{Discussion}
\label{sec:discussion}

The construction has two layers, and the interest lies in how they meet. The analytic layer embeds a finite system into $\mathbb{Z}_p$ through a choice of ordering and realizes its dynamics by rational approximants over $\mathbb{C}_p$ whose ball images have prescribed radii (Theorem~\ref{thm:existence_approximations}). The finite layer computes the radii, the coarse multipliers, the $A/E/I$ labels and the scores from $(f,\iota)$ (Proposition~\ref{thm:intrinsic_mu}), and Proposition~\ref{prop:radius_attained} is where the two agree: the combinatorial radii are the ones the rational maps realize. Since the classification moves with the embedding, the ordering becomes a variable of the problem, and Theorem~\ref{thm:muE_zero_lipschitz} identifies what minimizing over it is asking for.

\paragraph{Relation to prior work.}
The framework draws on three established lines, and it is against the first of them that its new content is measured. Non-Archimedean dynamics and automata theory use the inverse-limit picture $\mathbb{Z}_p=\varprojlim_k\mathbb{Z}/p^k\mathbb{Z}$ and balls as finite-resolution neighborhoods~\cite{Anashin2009,anashin2021automata_chapter,arrowsmith1994padic,khrennikov2004padicdynamics}. Theorem~\ref{thm:muE_zero_lipschitz} is what connects the present construction to that line: the class its vanishing locus singles out is the class of automaton functions~\cite{anashin2021automata_chapter}, and the new content is the graded version, a functional that measures the failure at each resolution instead of testing for it globally. $p$-adic biology uses ultrametric balls as observational equivalence classes in hierarchical state spaces~\cite{dragovich2021padicBio,khrennikov2010ultrametric,khrennikov2017protein}. Boolean-network theory treats reduced models through projections, latent variables, attractors and basin structure~\cite{shmulevich2002pbn,saadatpour2013boolean,naldi2011dynamically}, and ranks inputs inside each update function or groups nodes into modules~\cite{dimitrova2022canalizing,murrugarra2024modular}. Those methods operate on one update function or on a fixed grouping, while the ordering $\pi^*$ is global and dynamical, minimizing a scale-resolved functional over the joint map. A recent perspective on canalization~\cite{kadelka2026canalization} observes that beyond the Boolean case the notion of a small perturbation is itself model-dependent, since a Hamming step no longer carries a fixed biological meaning, and reports that no canonical replacement has emerged. The orderings supply a structured family of ultrametrics on the same finite state space, together with a principle for comparing them, which is one way of making that choice a variable rather than a convention.

\paragraph{Complexity.}
The cost $O(Np^N)$ of evaluating $\mu_E$ at one ordering presupposes that $f$ is given by its full transition table, and all the statements here are relative to that input model. Under it, evaluating $\mu_E(\pi)$ is polynomial in the input size, and deciding whether $\min_\pi\mu_E(\pi)=0$ lies in $\mathrm{P}$ by Corollary~\ref{cor:min_zero_dag}. The minimization itself is not harder: since $\mu_E$ depends on an ordering only through its chain of prefixes, $\min_\pi\mu_E(\pi)$ is a shortest-path problem on the subset lattice, and the reproducibility record uses that formulation as an independent exact cross-check without enumerating $S_N$. Under a succinct presentation of $f$, by update formulas or by a circuit, the size parameter is $N$ rather than $p^N$ and the complexity question is a different one, whose analysis needs its own input model. The chain cost is not submodular, so greedy approximation guarantees are not available from that route: Remark~\ref{rem:submodularity_negative} settles it in the negative with a counterexample on the network of Section~\ref{sec:athaliana}.

\paragraph{Other rankings of the coordinates.}
Ranking the coordinates of a finite system is a familiar aim, and the existing rankings are of a different kind. Network centrality orders them by the topology of the interaction graph~\cite{barabasi2004network}, fibration symmetry and input symmetry collapse coordinates whose local update data coincide~\cite{morone2020fibration,rozum2025cana}, and the isometries of the Boolean hypercube act on the state space by symmetries that leave the Hamming metric fixed~\cite{fabremonplaisir2021isometries}. Each returns a ranking or a symmetry at one scale. What $\mu_E$ adds is that the ordering is chosen by a functional defined across all the scales at once, and that the group acting is $S_N$ on the digit positions of the embedding, which changes the metric rather than preserving it. The degeneracy of $\mathrm{Opt}(f)$ recorded in Remark~\ref{rem:symmetry_C2C2} is the trace of that action.

\paragraph{Scope.}
The construction applies verbatim to finite systems over $\mathbb{F}_p^N$ for any prime $p$. Minimizing $\mu_E$ over $S_N$ is a combinatorial problem, certified here at $N=13$ by branch and bound.

The scores are indifferent to trajectory timing once $f$ is fixed. They read the images of $f$ across the ball hierarchy rather than any trajectory, so no notion of synchronous or asynchronous update enters them, and the fixed points that carry the $A/E/I$ word are maintained under any update order~\cite{bensussen2025analytical}. What does depend on the schedule is the transient structure. The convergence statement in the Supplementary Material is made for the synchronous update, and under an asynchronous schedule the same fixed configurations persist while their combinatorial basins and approach paths may change~\cite{bensussen2025analytical}.

\begin{remark}[The chain cost of $\mu_E$ is not submodular]
\label{rem:submodularity_negative}
The minimization $\pi^*\in\operatorname{arg\,min}_\pi\mu_E(\pi)$ runs over the finite group $S_N$, so the structure to look for is combinatorial. It would live in the dependence of $\mu_E$ on which coordinates occupy the coarsest positions. Theorem~\ref{lem:balls_are_blocks} makes this precise: writing, for a subset $J\subseteq\{0,\dots,N-1\}$,\[c(J)\;=\;p^{\,N-|J|}\cdot\#\{u\in\mathbb{F}_p^{J}\;:\;J\text{ is not autonomous at }u\},\]that is, $c(J)=p^{\,N-|J|}q_f(J)$ in the notation of Proposition~\ref{prop:intermediate_constant}, one has $\mu_E(\pi)=\sum_{n}c(J^\pi_n)$ over the prefix chain of $\pi$, so $\mu_E$ is the cost of a chain in the Boolean lattice under the set function $c$. The set function $c$ is not submodular. It already fails to be for the \textit{A.\ thaliana} map of Section~\ref{sec:athaliana}, where an explicit pair of coordinate subsets reverses the submodular inequality; the four values and their exhaustive verification are in the Supplementary Material, and the script that recomputes them from the transition table is in the reproducibility bundle~\cite{padic_grn_bundle}. Greedy approximation guarantees are therefore not available from submodularity of $c$, and what remains open is narrower and better posed: to characterize the maps for which $c$ \emph{is} submodular, or to find a different discrete convexity that the chain cost does satisfy. The two decision problems already separate: the vanishing question amounts to testing the dependency digraph for acyclicity, while $\min_\pi\mu_E(\pi)\le K$ with $K>0$ is settled here only at $N=13$, by branch and bound.
\end{remark}

\paragraph{Further questions.}
Three directions arise naturally from the construction and fall outside its aims. The first two were named above: characterizing the maps whose chain cost is submodular (Remark~\ref{rem:submodularity_negative}), and the complexity of minimizing $\mu_E$ when $f$ is presented succinctly, by circuits or by Boolean update rules. Mixed-radix state spaces, where the alphabet size varies with the coordinate, call for a profinite or adelic treatment rather than a single prime. A related limitation is that the distance used here grades a perturbation by which coordinate it touches and how deep that coordinate sits, and not by how far the value moves within it: over $\mathbb{F}_3$ the digits $0$, $1$ and $2$ are pairwise equidistant. Ordering the levels of a single coordinate is a second axis, and combining the two would need either a finer alphabet or a restriction of the admissible orderings. The results developed here are independent of all three, and each merits a separate treatment.

Every finite system placed on $\mathbb{Z}_p$ by an ordering carries a graded distance to the automaton class, and the ordering that minimizes it belongs to the system and not to the way the system happened to be written down.

%% file: main-preprint.bbl
\begin{thebibliography}{10}

\bibitem{anashin2021automata_chapter}
Vladimir Anashin.
\newblock The $p$-adic theory of automata functions.
\newblock In W.~A. Z\'u\~niga Galindo and B.~Toni, editors, {\em Advances in
  Non-Archimedean Analysis and Applications: The $p$-adic Methodology in
  {STEAM-H}}, STEAM-H: Science, Technology, Engineering, Agriculture,
  Mathematics \& Health, pages 9--113. Springer, Cham, 2021.

\bibitem{Anashin2009}
Vladimir Anashin and Andrei Khrennikov.
\newblock {\em Applied Algebraic Dynamics}, volume~49 of {\em De Gruyter
  Expositions in Mathematics}.
\newblock Walter de Gruyter, Berlin, 2009.

\bibitem{arrowsmith1994padic}
David~K. Arrowsmith and Franco Vivaldi.
\newblock Geometry of \texorpdfstring{$p$}{p}-adic {Siegel} discs.
\newblock {\em Physica D: Nonlinear Phenomena}, 71(1-2):222--236, 1994.

\bibitem{barabasi2004network}
Albert-L{\'a}szl{\'o} Barab{\'a}si and Zoltan~N Oltvai.
\newblock Network biology: Understanding the cell's functional organization.
\newblock {\em Nature Reviews Genetics}, 5(2):101--113, 2004.

\bibitem{benedetto2019dynamics}
Robert~L Benedetto.
\newblock {\em Dynamics in one non-archimedean variable}, volume 198 of {\em
  Graduate Studies in Mathematics}.
\newblock American Mathematical Society, Providence, RI, 2019.

\bibitem{bensussen2025analytical}
Antonio Bensussen, J.~Arturo Arciniega-Gonz\'alez, Elena~R. \'Alvarez-Buylla,
  and Juan~Carlos Mart\'inez-Garc\'ia.
\newblock Analytical approach of synchronous and asynchronous update schemes
  applied to solving biological {Boolean} networks.
\newblock {\em PLOS One}, 20(9):e0319240, 2025.

\bibitem{chaos2006genes}
\'Alvaro Chaos, Max Aldana, Carlos Espinosa-Soto, Berenice Garc\'ia Ponce~de
  Le\'on, Adriana Garay~Arroyo, and Elena~R. \'Alvarez-Buylla.
\newblock From genes to flower patterns and evolution: dynamic models of gene
  regulatory networks.
\newblock {\em Journal of Plant Growth Regulation}, 25(4):278--289, 2006.

\bibitem{chavez2022flowering}
Elva~C. Ch{\'a}vez-Hern{\'a}ndez, Stella Quiroz, Berenice Garc{\'\i}a-Ponce,
  and Elena~R. {\'A}lvarez-Buylla.
\newblock The flowering transition pathways converge into a complex gene
  regulatory network that underlies the phase changes of the shoot apical
  meristem in \textit{{A}rabidopsis thaliana}.
\newblock {\em Frontiers in Plant Science}, 13:852047, 2022.

\bibitem{cortes2026forest}
Yuriria Cort\'es-Poza and J.~Rogelio P\'erez-Buend\'ia.
\newblock Topological structure of epigenetic forests in flower morphogenesis.
\newblock {\em Bulletin of Mathematical Biology}, 88(7):124, 2026.

\bibitem{dimitrova2022canalizing}
Elena Dimitrova, Brandilyn Stigler, Claus Kadelka, and David Murrugarra.
\newblock Revealing the canalizing structure of {Boolean} functions:
  {Algorithms} and applications.
\newblock {\em Automatica}, 146:110630, 2022.

\bibitem{dragovich2021padicBio}
Branko Dragovich, Andrei~Yu. Khrennikov, Sergei~V. Kozyrev, and
  Nata{\v{s}}a~{\v{Z}}. Mi{\v{s}}i{\'c}.
\newblock p-adic mathematics and theoretical biology.
\newblock {\em Biosystems}, 199:104288, 2021.

\bibitem{espinosa2004gene}
C.~Espinosa-Soto, P.~Padilla-Longoria, and E.~R. \'Alvarez-Buylla.
\newblock A gene regulatory network model for cell-fate determination during
  \textit{Arabidopsis thaliana} flower development that is robust and recovers
  experimental gene expression profiles.
\newblock {\em The Plant Cell}, 16(11):2923--2939, 2004.

\bibitem{fabremonplaisir2021isometries}
Jean Fabre-Monplaisir, Brigitte Moss\'e, and \'Elisabeth Remy.
\newblock Isometries of the hypercube: A tool for {Boolean} regulatory networks
  analysis.
\newblock {\em Physica D: Nonlinear Phenomena}, 424:132831, 2021.

\bibitem{hedlund1969endomorphisms}
G.~A. Hedlund.
\newblock Endomorphisms and automorphisms of the shift dynamical system.
\newblock {\em Mathematical Systems Theory}, 3(4):320--375, 1969.

\bibitem{kadelka2026canalization}
Claus Kadelka.
\newblock Canalization as a stabilizing principle of gene regulatory networks:
  a discrete dynamical systems perspective.
\newblock {\em npj Systems Biology and Applications}, 12:30, 2026.

\bibitem{kadelka2024metaanalysis}
Claus Kadelka, Taras-Michael Butrie, Evan Hilton, Jack Kinseth, Addison
  Schmidt, and Haris Serdarevic.
\newblock A meta-analysis of {B}oolean network models reveals design principles
  of gene regulatory networks.
\newblock {\em Science Advances}, 10(2):eadj0822, 2024.

\bibitem{khrennikov2010ultrametric}
Andrei Khrennikov.
\newblock Modelling of psychological behavior on the basis of ultrametric
  mental space: Encoding of categories by balls.
\newblock {\em p-Adic Numbers, Ultrametric Analysis and Applications},
  2(1):1--20, 2010.

\bibitem{khrennikov2017protein}
Andrei Khrennikov and Ekaterina Yurova.
\newblock Automaton model of protein: Dynamics of conformational and functional
  states.
\newblock {\em Progress in Biophysics and Molecular Biology}, 130:2--14, 2017.

\bibitem{khrennikov2004padicdynamics}
Andrei~Yu. Khrennikov and Marcus Nilsson.
\newblock {\em p-Adic Deterministic and Random Dynamics}, volume 574 of {\em
  Mathematics and Its Applications}.
\newblock Kluwer Academic Publishers, Dordrecht, 2004.

\bibitem{leoncardenal2022lipschitz}
Edwin Le\'on-Cardenal, John~Jaime Rodr\'iguez-Vega, and Fausto
  Bol\'ivar-Barbosa.
\newblock {A New Class of $p$-Adic Lipschitz Functions and Multidimensional
  Hensel's Lemma}.
\newblock {\em p-Adic Numbers, Ultrametric Analysis and Applications},
  14(1):1--13, 2022.

\bibitem{rozum2025cana}
Austin~M. Marcus, Jordan Rozum, Herbert Sizek, and Luis~M. Rocha.
\newblock {CANA} v1.0.0: efficient quantification of canalization in automata
  networks.
\newblock {\em Bioinformatics}, 41(10):btaf461, 2025.

\bibitem{mendoza1998dynamics}
Luis Mendoza and Elena~R Alvarez-Buylla.
\newblock Dynamics of the genetic regulatory network for arabidopsis thaliana
  flower morphogenesis.
\newblock {\em Journal of Theoretical Biology}, 193(2):307--319, 1998.

\bibitem{morone2020fibration}
Flaviano Morone, Ian Leifer, and Hern{\'a}n~A. Makse.
\newblock Fibration symmetries uncover the building blocks of biological
  networks.
\newblock {\em Proceedings of the National Academy of Sciences},
  117(15):8306--8314, 2020.

\bibitem{murrugarra2024modular}
David Murrugarra, Alan Veliz-Cuba, Elena Dimitrova, Claus Kadelka, Matthew
  Wheeler, and Reinhard Laubenbacher.
\newblock Modular control of {Boolean} network models.
\newblock {\em Bulletin of Mathematical Biology}, 87(7):91, 2025.

\bibitem{naldi2011dynamically}
Aur{\'e}lien Naldi, Elisabeth Remy, Denis Thieffry, and Claudine Chaouiya.
\newblock Dynamically consistent reduction of logical regulatory graphs.
\newblock {\em Theoretical Computer Science}, 412(21):2207--2218, 2011.

\bibitem{rogelio2023gluing}
V{\'\i}ctor Nopal~Coello and J.~Rogelio P{\'e}rez-Buend{\'\i}a.
\newblock Gluing dynamics: $\varepsilon$-precision in solving a non-archimedean
  inverse problem.
\newblock {\em Bolet{\'\i}n de la Sociedad Matem{\'a}tica Mexicana}, 31(2):52,
  2025.

\bibitem{perez2026interpreters}
J.~Rogelio P\'erez-Buend\'ia.
\newblock {Rational Interpreters for Discrete Dynamics: Existence, Exactness,
  and Decomposition over $p$-adic Fields}.
\newblock Preprint, arXiv:2602.05433 [math.DS], 2026.
\newblock \url{https://doi.org/10.48550/arXiv.2602.05433}.

\bibitem{perez2022epigenetic}
J.~Rogelio P\'erez-Buend\'ia, Y.~Cort\'es-Poza, and P.~Padilla-Longoria.
\newblock Epigenetic forest and flower morphogenesis.
\newblock {\em Computational Biology and Chemistry}, 98:107667, 2022.

\bibitem{padic_grn_bundle}
J.~Rogelio P\'erez-Buend\'ia and V\'{\i}ctor Nopal-Coello.
\newblock {Reproducibility bundle: Hierarchical Expansion of Finite Discrete
  Dynamical Systems. A Non-Archimedean Variational Theory over Coordinate
  Orderings}.
\newblock Zenodo, 2026.
\newblock \url{https://doi.org/10.5281/zenodo.20709159}. Transition table,
  analysis code, exact search over orderings, and verification scripts.

\bibitem{perez2026padic_applications}
J.~Rogelio P\'erez-Buend\'ia, V\'{\i}ctor Nopal-Coello, and Yuriria
  Cort\'es-Poza.
\newblock {Regulatory $p$-Adic Orderings from Stability in Gene Networks:
  Applications and Validation in Three Developmental Systems}.
\newblock Preprint, Zenodo, 2026.
\newblock \url{https://doi.org/10.5281/zenodo.19426870}.

\bibitem{riveraletelier2003dynamique}
Juan Rivera-Letelier.
\newblock Dynamique des fonctions rationnelles sur des corps locaux.
\newblock {\em Ast\'erisque}, (287):147--230, 2003.
\newblock Doctoral thesis, Universit\'e Paris-Sud (Orsay), 2000.

\bibitem{saadatpour2013boolean}
Assieh Saadatpour, R\'eka Albert, and Timothy~C. Reluga.
\newblock A reduction method for {Boolean} network models proven to conserve
  attractors.
\newblock {\em SIAM Journal on Applied Dynamical Systems}, 12(4):1997--2011,
  2013.

\bibitem{shmulevich2002pbn}
Ilya Shmulevich, Edward~R. Dougherty, Seungchan Kim, and Wei Zhang.
\newblock Probabilistic boolean networks: a rule-based uncertainty model for
  gene regulatory networks.
\newblock {\em Bioinformatics}, 18(2):261--274, 2002.

\bibitem{zunigagalindo2025padicbook}
W.~A. {Z{\'u}{\~n}iga-Galindo}.
\newblock {\em $p$-Adic Analysis: Stochastic Processes and Pseudo-Differential
  Equations}, volume~11 of {\em Advances in Analysis and Geometry}.
\newblock De Gruyter, Berlin, 2024.

\bibitem{Zuniga-Galindo2022}
Wilson~A. {Z{\'u}{\~n}iga-Galindo}.
\newblock p-adic analysis: A quick introduction.
\newblock In Carlos Galindo, Alejandro Melle~Hern{\'a}ndez, Julio
  Moyano-Fern{\'a}ndez, and Wilson {Z{\'u}{\~n}iga-Galindo}, editors, {\em
  p-Adic Analysis, Arithmetic and Singularities}, volume 778 of {\em
  Contemporary Mathematics}, pages 177--221. American Mathematical Society,
  Providence, RI, 2022.
\newblock Mini-Course, L. Santal\'o Research Summer School 2019, Santander,
  Spain.

\end{thebibliography}
